\documentclass{article}
\usepackage{amsmath,amsfonts,amssymb,amsthm,cases,graphicx,hyperref,color}

\usepackage{tikz}

\usepackage{vmargin}
\setmarginsrb{2cm}{1cm}{2cm}{3cm}{1cm}{1cm}{2cm}{2cm}
\usepackage{authblk}

\renewcommand{\leq}{\leqslant}

\renewcommand{\geq}{\geqslant}

\numberwithin{equation}{section}
\everymath{\displaystyle}

\newcommand{\C}{{\cal C}}
\newcommand{\R}{\mathbb{R}}

\newcommand{\norm}[1]{\ensuremath{\left\|#1\right\|}}
\newcommand{\abs}[1]{\ensuremath{\left|#1\right|}}

\newcommand{\dx}{\ensuremath{\, dx}}

\newcommand{\dt}{\ensuremath{\, dt}}

\DeclareMathOperator{\divg}{div}

\newtheorem{thm}{Theorem}
\newtheorem{lem}{Lemma}
\newtheorem{prop}{Proposition}

\newtheorem{remark}{Remark}

\title{Mixed Finite Element Method and numerical analysis of a convection-diffusion-reaction model in a porous medium}
\author[1]{\'Elo\"ise Comte}
\affil[1]{La Rochelle Universit\'e, MIA, 23 Avenue A. Einstein, BP 33060, 17031 La Rochelle}

\date{June 2019}

\begin{document}

\maketitle

\abstract{A hydrogeological model for the spread of pollution in an aquifer is considered. The model consists in a convection-diffusion-reaction equation involving the dispersion tensor which depends nonlinearly of the fluid velocity. We introduce an explicit flux in the model and use a mixed Finite Element Method for the discretization. We provide existence, uniqueness and stability results for the discrete model. A convergence result is obtained for the semi-discretized in time problem and for the fully discretization.  }

\vspace{1cm}


\section{Description of the model}
\label{sec:themodel}
The model describes the transport of a pollutant in the groundwater. For the lake of simplicity, we consider that there is only one pollutant of interest. The latter will be modeled by a reaction-convection-diffusion equation. Besides the pollutant concentration, the velocity is also a variable of the model, expressed as a function of the hydraulic head through the Darcy law. 

Let $\Omega \subset \R^N$, $N\leq 3$, a bounded domain with a boundary $\partial \Omega \in \C^2$. Time horizon is denoted by $T$, with $0 < T \leq \infty$. Let $\Omega_T = \Omega \times (0,T)$. 

Denoted by $p$ the pollutant load, $c$ its concentration, $v$ the velocity of the mixture and $\phi$ the hydraulic head, the model is the following PDE's system: 
\begin{eqnarray}
\label{eqetat}
R \psi \partial_t c  -(\kappa \nabla \phi)\cdot \nabla c - \divg(S(v)\psi \nabla c) = -r(c) - g c + p\chi_{S}  \text{ in $\Omega_T$,} 
\end{eqnarray}
\begin{eqnarray}
\label{eqetatphi}
\divg(v) =g  \text{, $v=-\kappa \nabla \phi$ in $\Omega_T$,} 
\end{eqnarray}
where function $\chi_{\mathcal{S}}$ is the indicatrice of a part $\mathcal{S}$ of $\bar{\Omega}$. 

We assume that the eventual adsorption of the pollutant by the soil is linear and instantaneous reaction (see the arguments in de Marsily \cite{Marsily}, page 251). Coefficient $R$ is the corresponding retardation factor. 

The soil porosity is described by function $\psi$. The structure of the soil is also described in the fluid mobility tensor, $\kappa$,  rating the permeability of the underground with the viscosity of the fluid. 
The dispersion tensor is denoted by $S(v)$. 
Following Scheidegger \cite{Scheidegger1974}, we consider the nonlinear dependency of the longitudinal and transverse components of the dispersion on the velocity: tensor $S(v)$ is such that 
\begin{eqnarray}  
\label{es:coeffdiffusion}
S(v)=S_m \mathrm{Id}+ S_p(v), \quad S_p(v) = \left|v\right| \Bigl(\frac{\alpha_L}{\vert v\vert^2} v \otimes v + \alpha_T
( \mathrm{Id}-\frac{1}{\vert v\vert^2} v \otimes v ) \Bigr)
\end{eqnarray}
where $S_m$, $\alpha_L$ and $\alpha_T$ are respectively the diffusion coefficient, the longitudinal and transverse dispersion factors.
Here $u \otimes v$ denotes the tensor product, $(u \otimes v)_{ij}=u_i v_j$, while $u \cdot v$ denotes the scalar product, $u \cdot 
v= \sum_{i=1}^N u_i v_i$ and $\vert u \vert^2= u \cdot u$. The identity matrix is denoted by $\mathrm{Id}$.

The result of the reaction of the pollutant, transformed it in other miscible species, is modeled by $r(c)$. Classical isotherms are described by linear functions, in the form $r(c)=kc$ of by Freundlich function $r(c)=kc^{k'}$, or by Langmuir function $r(c)=kc/(1+k^{'} c)$, $(k,k^{'}) \in \R^2_+$. 

The other source term $g$ takes into account the contribution from the soil itself and other inputs. 

The model \eqref{eqetat}-\eqref{eqetatphi} is completed by the following initial and boundary conditions: 
\begin{eqnarray}
\label{condetat}
\chi S(v) \nabla c \cdot n + (1-\chi)c = 0 \text{ on $\partial \Omega \times (0,T)$, $c{|_{t=0}}=c_0$ in $\Omega$,} 
\end{eqnarray}
with $\chi=0$ or $1$ if we use Dirichlet or Neumann boundary conditions, and
\begin{eqnarray}
\label{condetatphi}
\tilde{\chi} \kappa \nabla \phi \cdot n +(1-\tilde{\chi})\phi = 0 \text{ on $\partial \Omega \times (0,T)$,} 
\end{eqnarray}
with $\tilde{\chi}=0$ or $1$. 

In order to avoid futur problems from the convective term in the numerical tests, we establish a more physical formulation, by introducing the flux explicitly given by $v_c = -S(v) \psi \nabla c + vc$. 
We use the formulation $\divg (v c) = c \divg (v) + v \cdot \nabla c$ and $\divg(v)=g$, to get: 
\begin{eqnarray}
\label{eqetatbis}
R \psi \partial_t c + \divg (v_c) = -r(c)  + p\chi_{S} \text{ in $\Omega_T$,} 
\end{eqnarray}
with
\begin{eqnarray}
\label{vc}
v_c=-S(v) \psi \nabla c + v c,
\end{eqnarray}
completed by following initial and boundary conditions: 
\begin{eqnarray}
\label{condetatbis}
\chi S(v) \nabla c \cdot n +(1-\chi)c = 0 \text{ on $\partial \Omega \times (0,T)$, $c{|_{t=0}}=c_0$ in $\Omega$.} 
\end{eqnarray}
Let
\begin{eqnarray*} 
\tilde{p}=
\begin{cases} 
p \text{ if $x \in \mathcal{S}$,} \\ 
0 \text{ else.} 
\end{cases}
\end{eqnarray*}
By commodity, we will denote $\tilde{p}=p$.

\section{Assumptions and preliminary results} \label{sec:results}

\subsection{Assumptions} \label{subsec:assumptions}
We consider the following assumptions:
\begin{itemize}
\item $S(v)$ is supposed to be given by \eqref{es:coeffdiffusion}. According to the definition \eqref{es:coeffdiffusion} of the dispersion tensor, we get 
\begin{eqnarray}
\label{eq5.6bis}
S(v) \xi \cdot \xi \geq (S_m + \alpha_T \abs{v}) \abs{\xi}^2, \abs{S(v)\xi} \leq (S_m + \alpha_L \abs{v}) \abs{\xi} \text{ $\forall \xi \in \R^N$}.
\end{eqnarray}
We assume that $$S_m > 0, \alpha_L > \alpha_T \geq 0.$$ 
The case without dispersion, \textit{i.e.} $\alpha_L=\alpha_T=0$, is treated by Kumar et al \cite{Kumar}. We assume that there exist reals $\kappa_{-}$ and $\kappa_{+}$ with $0 < \kappa_{-} \leq \kappa_{+}$ such that the permeability tensor satisfies:
\begin{eqnarray*}
\kappa \xi \cdot \xi \geq \kappa_{-} \abs{\xi}^2 \text{ and } \abs{\kappa \xi} \leq \kappa_{+} \abs{\xi}, \text{ $\forall \xi \in \R^N$ ;}
\end{eqnarray*}
\item isotherm $r$ is $\C^1$, concave, derivable with a bounded derivate on $\R_+$ and such that there exists $r_{+} \in \R^+$ $$ \abs{r(x)} \leq r_{+} \abs{x}, \forall x \in \R  \text{ ;}$$
\item the retardation factor $R$ is strictly positive; 
\item the initial concentration is $c_0 \in H^1(\Omega)$ with $0 \leq c_0 \leq 1$ a.e. in $\Omega$ ; 
\item function $g$ is $L^{\infty}(\Omega_T)$ ;
\item porosity $\psi$ is $L^{\infty}(\Omega)$ and there exist reals $\psi_{-}$ and $\psi_{+}$ such that
$$
0 < \psi_{-} \leq \psi(x) \leq \psi_{+} \text{, a.e. } x \in \Omega.
$$
\end{itemize} 
\vspace{2mm}
We completed the assumptions on the dispersion tensor $S(v)$ given by \eqref{es:coeffdiffusion}: it is symmetrical and 
\begin{eqnarray}
\label{eq5.7.0} 
S(v)^{-1} \xi \cdot \xi \geq \frac{1}{S_m + \alpha_T \abs{v}} \Bigl( 1 - \frac{(\alpha_L - \alpha_T) \abs{v}}{S_m+\alpha_T \abs{v}} \Bigr) \abs{\xi}^2 \geq M_{-} \abs{\xi}^2 \text{, } M_{-} > 0 
\end{eqnarray}
for all $\xi \in \R^N$, $N=2$ or $3$. Indeed, according to \eqref{es:coeffdiffusion}, $S(v)$ follows
$$S(v)=(S_m + \alpha_T \abs{v})(Id + E) \text{, with $E=\frac{(\alpha_L - \alpha_T) \abs{v}}{S_m + \alpha_T \abs{v}} \epsilon(v)$},$$
the inverse $(Id + E)^{-1}=Id - E(Id+E)^{-1}$ being well defined, matrix $Id$ and $E$ being positive-definite. \\
According to the symmetrical and the positivity of $(Id+E)$, we can defined $(Id+E)^{\frac{1}{2}}$ and $(Id+E)^{-\frac{1}{2}}$. We get 
$$(Id+E)^{-1} \xi \cdot \xi = \abs{\xi}^2 - E(Id+E)^{-\frac{1}{2}} \xi \cdot (Id + E)^{-\frac{1}{2}} \xi $$
where $E$ is such that $-E \eta \cdot \eta \geq - \frac{(\alpha_L - \alpha_T) \abs{v}}{S_m + \alpha_T \abs{v}} \abs{\eta}^2$ for all $\eta \in \R^N$. Applying this latter result to $\eta = (Id+E)^{-\frac{1}{2}} \xi$, which is such that $\abs{\eta} \leq \xi$, we get 
$$(Id+E)^{-1} \xi \cdot \xi \geq \Bigl( 1 - \frac{(\alpha_L - \alpha_T) \abs{v}}{S_m + \alpha_T \abs{v}} \Bigr) \abs{\eta}^2.$$
Noticing $(S(v))^{\frac{1}{2}}$ the unique symmetrical positive-definite matrix such that 
$$S(v)^{\frac{1}{2}} S(v)^{\frac{1}{2}}=S(v).$$
We deduce from \eqref{eq5.6bis} that there exists $C \geq 0$ such that
\begin{eqnarray}
\label{eq5.7}
\abs{S(v)^{\frac{1}{2}} \xi} \leq C(1+ \abs{v}^{\frac{1}{2}})\abs{\xi} \text{, for all $\xi \in \R^N$.}
\end{eqnarray}
Furthermore, there exists a constant $M_+ >0$ such that
\begin{eqnarray}
\label{eq5.7.1}
\abs{S(v)^{-1} v} \leq M_+.
\end{eqnarray}
Indeed, let $\xi \in \R^N$, $\xi \neq 0_{\R^N}$. The operator $S(v)$ being inversible, let $\mu = S(v)^{-1} \xi$ and $\mu \neq 0_{\R^N}$. We get $\abs{S(v) \mu \cdot \mu} = \abs{\xi \cdot \mu} \leq \abs{\xi} \abs{\mu}$. Otherwise, by definition of $S(v)$, 
\begin{align*}
\abs{S(v)\mu \cdot \mu}&=S(v)\mu \cdot \mu = (S_m + \alpha_T \abs{v}) \abs{\mu}^2 + \frac{\alpha_L - \alpha_T}{\abs{v}} \abs{v \cdot \mu}^2 &&\\
&= (S_m + \alpha_T \abs{v}) \abs{\mu}^2 +  \frac{\alpha_L - \alpha_T}{\abs{v}} \frac{\abs{v \cdot \mu}^2}{\abs{\mu}^2} \abs{\mu}^2 &&\\
&\leq \abs{\xi} \abs{\mu}.&&
\end{align*}
Simplifying by $\abs{\mu}$, and keeping in mind that $\abs{\mu} = \abs{S(v)^{-1} \xi}$, we get: 
\begin{eqnarray}
\label{eq.etoile}
\abs{S(v)^{-1}\xi} \leq \frac{1}{S_m + \alpha_T \abs{v} + \frac{\alpha_L - \alpha_T}{\abs{v}} \frac{\abs{v \cdot \mu}^2}{\abs{\mu}^2}} \abs{\xi}.
\end{eqnarray}
We now estimated $\displaystyle \frac{\abs{v \cdot \mu}^2}{\abs{\mu}^2}$ in the case where $\xi = v$. In this case, $\mu \cdot v = S(v)^{-1} v \cdot v$. From \eqref{eq5.7.0}, $\abs{\mu \cdot v} = \mu \cdot v$ and
$$\abs{\mu \cdot v} \geq M_- \abs{v}^2.$$
Furthermore, from \eqref{eq.etoile} we get $\displaystyle \abs{S(v)^{-1} v} = \abs{\mu} \leq \frac{\abs{v}}{S_m}$. Thus, 
$$\frac{\abs{\mu \cdot v}^2}{\abs{\mu}^2} \geq M_-^2 S_m^2 \abs{v}^2.$$
Using this estimation in \eqref{eq.etoile}, we get
$$\abs{S(v)^{-1}v} \leq \frac{\abs{v}}{S_m + \alpha_T \abs{v} +(\alpha_L - \alpha_T) M_-^2 S_m^2 \abs{v}}$$
and there exists $M_+ > 0$ such that for all $v \in \R^N$, 
$$ \frac{\abs{v}}{S_m + \alpha_T \abs{v} +(\alpha_L - \alpha_T) M_-^2 S_m^2 \abs{v}} \leq M_+.$$
We will use these estimations \eqref{eq5.7} et \eqref{eq5.7.1} in order to control some terms in the form of $\langle S(v)^{-1} v u, S(v)^{\frac{1}{2}} w \rangle$ with $u \in R$, $w \in R^N$. We get
\begin{align*}
\abs{\langle S(v)^{-1} v u, S(v)^{\frac{1}{2}} w \rangle} &\leq M_+ \abs{u} \abs{S(v)^{\frac{1}{2}} w } &&\\
&\leq M_+ C \Bigl(1+ \abs{v}^{\frac{1}{2}}\Bigr) \abs{u} \abs{w}&&
\end{align*}
thus
\begin{eqnarray}
\label{eq5.7.2}
\abs{\langle S(v)^{-1} v u, S(v)^{\frac{1}{2}} w \rangle}  \leq C_{disp} \Bigl(1+ \abs{v}^{\frac{1}{2}}\Bigr) \abs{u} \abs{w}.
\end{eqnarray}

\subsection{Preliminary results} \label{subsec:prelresults}

\begin{prop} \label{prop:phi} 
There exits a unique function $\phi \in L^{\infty }(0,T;H^{1}(\Omega ))$ solving
problem \eqref{eqetatphi}-\eqref{condetatphi}.
Moreover, for $\tilde{\chi}=0$ (Dirichlet boundary conditions), we have: 
\begin{enumerate}
\item If $\phi_{1}\in W^{2,p}(\bar{\Omega})$ with $p>N$ and $\kappa \in (\mathcal{C}^{1}(\bar{\Omega}))^{N\times N}$, then $\phi $  belongs to $L^{\infty
}(0,T;W^{2,p}(\Omega ))$.
In particular,  $v=-\kappa \nabla \phi$ belongs to $L^{\infty }(\Omega _{T})$.
\item If $\phi _{1}\in W^{2,p}(\bar{\Omega}) $ with $p>N/2$, then $\phi $  belongs to $L^{\infty}(\Omega_T)$.
If moreover $\kappa =\kappa^* Id$ with $\kappa^* :\bar{\Omega}\rightarrow {\mathbb{R}}$, $\kappa^* \in C^{1}(\bar{\Omega})$, then  $\phi $  belongs to $L^{\infty}(0,T;W^{1,q}(\Omega ))$ with $q>N$.
In particular,  $v=-\kappa \nabla \phi$ belongs to $L^{\infty }(\Omega _{T})$.
\end{enumerate}
\end{prop}
\begin{proof}
We refer the reader to the proof of Proposition 3.1 in Augeraud-V\'eron et al. \cite{aug2017}.
\end{proof}

\section{Mixed Finite Element Method}
The variational formulation of the problem \eqref{eqetat}-\eqref{eqetatphi} with \eqref{condetat}-\eqref{condetatphi} is the following: finding $(c,\mu,\phi)$, with $c \in L^2(0,T;H^1(\Omega))$, $\phi \in L^{\infty}(0,T;W^{2,p}(\Omega))$ with $p>N$ or $\phi \in L^{\infty}(0,T;W^{1,q}(\Omega))$ with $q>N$ (and then $v \in L^{\infty}(\Omega)$), such that 
\begin{align*}
\int_{\Omega} R \psi \partial_t c  \varphi \dx - \int_{\Omega} (\kappa \nabla \phi \cdot \nabla c) \varphi \dx &+ \int_{\Omega} \psi S(v)  \nabla c \cdot \nabla \varphi \dx =&&\\
& \int_{\Omega} \Bigl(-r(c) -gc + p\chi_{\mathcal{S}} \Bigr) \varphi \dx, &&\\
&\int_{\Omega} \kappa \nabla \phi \cdot \nabla \tilde{\varphi }\dx = \int_{\Omega}  g \tilde{\varphi} \dx,&&
\end{align*}
for all test functions $\varphi$ and $\tilde{\varphi}$ such that
\begin{eqnarray}
\label{testphi}
\varphi \in L^2(0,T;H^1(\Omega)) \cap H^1(0,T;L^2(\Omega)) \text{ with } (1-\chi)\varphi \in L^2(0,T;H^1_0(\Omega)), 
\end{eqnarray}
\begin{eqnarray}
\label{testphitilde}
\tilde{\varphi} \in 
L^2(0,T;H^1(\Omega)) \text{ with } (1- \tilde{\chi})\tilde{\varphi} \in L^2(0,T;H^1_0(\Omega)).
\end{eqnarray}
We denoted
$$W=\{ f \in H^1(\Omega), (1-\chi)f \in H_0^1(\Omega) \}$$ 
$$\tilde{W}= \{ f \in H^1(\Omega), (1- \tilde{\chi})f \in H^1_0(\Omega) \}.$$ 
The model will be discretized in time and in space. Lets start by the spacial discretization.  We define $\mathcal{T}_h$ a regular decomposition of $\Omega \subset \R^N$ into $N$ closed simplices where $h$ stands for the mesh diameter. \\
We assume that $\bar{\Omega} = \cup_{\mathcal{O} \in \mathcal{T}_h}\mathcal{O}$. We define the following discrete subspaces (finit elements)
$W_h \subset W$, $\tilde{W}_h \subset \tilde{W}$, by
 \begin{eqnarray}
 \label{Wh}
 W_h := 
\{ q \in W, q \text{ is constant on each mesh $\mathcal{O} \in \mathcal{T}_h$}\} 
\end{eqnarray}
 \begin{multline}
 \label{Whtilde}
 \tilde{W}_h := \{ q \in \tilde{W}, q \text{ is at most first degree polynomial on each element of $\mathcal{O} \in \mathcal{T}_h$}\}. 
\end{multline}
For the vectorial functions, we define: 
\begin{multline*}
U_{h} := \{ q \in H(\divg,\Omega), \forall i \in \{ 1,...,N \}, \text{ the i-th componant of $q$ is a polynomial of }\\ \text{ degree $\leq 1$ on each element of $\mathcal{O} \in \mathcal{T}_h$} \}
\end{multline*}
and
\begin{multline*}
U_{h,0} := \{ q \in H(\divg,\Omega) \text{ with $q \cdot n =0$ sur $\partial \Omega$ } \forall i \in \{ 1,...,N \},   \text{ the i-th composant of $q$ }\\ \text{ is a polynomial of degree $\leq 1$ on each element of $\mathcal{O} \in \mathcal{T}_h$} \}.
\end{multline*} 
For the time discretization, we denoted by $\tau$ the time step. A semi-implicit Euler scheme will be use. 

\subsection{Discrete hydraulic head problem analysis} \label{subsec:hydraulic}
The problem is the following: \\
\\
 \noindent\fbox{\parbox{\linewidth \fboxrule \fboxsep}{
Find $\phi_h \in \tilde{W}_h$ and $v_h \in U_h$ (or $U_{h,0}$ according to the boundary conditions) such that
\begin{eqnarray}
\label{discretphi}
\begin{cases} \displaystyle{
 \int_{0}^{T} \int_{\Omega} v_h \cdot \psi_h \dx \dt + \int_{0}^{T} \int_{\Omega} \kappa \nabla \phi_h \cdot  \psi_h \dx \dt = 0, }\\
\displaystyle{ \int_{0}^{T} \int_{\Omega} v_h \cdot \nabla \varphi_h \dx \dt + \int_{0}^{T} \int_{\Omega} g \varphi_h \dx \dt = 0, }
 \end{cases}
 \end{eqnarray}
for all $\psi_h \in U_h$ (or $U_{h,0}$) and $\varphi_h \in W_h$.}} \\
\\
The analysis of this problem is classical both for Neumann or Dirichlet boundary conditions (see the first schema of Achdou et al. \cite{Bernardi} or Girault et al. \cite{Girault}). The variational formulation \eqref{discretphi} allowing to compute $\phi_h$ does not need a space discretisation. The benefit is that this scheme can easily be used for highly contrasted media (Achdou and Bernardi \cite{Achdou}). Moreover, we get the following optimal convergence result: 

\begin{prop}
\label{prop:convBernardi}
(i) If the solution of the problem \eqref{eqetatphi}-\eqref{condetatphi} is such that $(v,\phi) \in (H^s(\Omega))^N \times H^{s+1}(\Omega)$, for a $s \in \R$ such that $0 < s \leq 1$, then we get the following optimal uniform estimation: 
$$ \norm{v - v_h}_{(L^2(\Omega))^N} + \norm{\phi - \phi_h}_{H^1(\Omega)} \leq C h^s \Bigl( \norm{v}_{(H^s(\Omega))^N} + \norm{\phi}_{H^{s+1}(\Omega)} \Bigr).$$
(ii) In particular, if one of the following assumptions is satisfied, 
\begin{itemize}
\item $\partial \Omega$ is $\C^2$, $\kappa \in (\C^1(\Omega))^{N\times N}$, 
\item $\partial \Omega$ is $\C^2$, $\kappa = \kappa^* Id$ with $\kappa^* : \bar{\omega} \rightarrow \R$ and $\kappa \in \C^1(\Omega)$, 
\end{itemize}
we know (according to Proposition \ref{prop:phi}) that $\phi \in L^{\infty}(0,T;W^{2,p}(\Omega))$ and $v\in (L^{\infty}(\Omega_T))^N$. We thus get the strong convergence of the scheme in $L^{\infty}(0,T;L^2(\Omega))$ in the sense where 
$$ \lim \limits_{h \rightarrow 0} \norm{v - v_h}_{(L^2(\Omega_T))^N} = 0. $$
Moreover, for all $h$, 
$$ \norm{v_h}_{(L^{\infty}(\Omega_T))^N} \leq C.$$ 
\end{prop}
\begin{proof}
Point (i) is theorem 8 of Achdou et al. \cite{Bernardi}. Under the assumptions of (ii), by choosing $p=2$, point (i) applies and gives $\lim \limits_{h \rightarrow 0} \norm{v_h -v}_{(L^2(\Omega_T))^N} = 0$. The same convergence is then also true almost everywhere in $\Omega_T$. Thus, $\forall \epsilon > 0$, there exists $H > 0$ such that if $h<H$, 
$$ \abs{v_{h,i} - v_{,i}} < \epsilon \text{, $1\leq i \leq N$, almost everywhere in $\Omega_T$}$$
where we denoted $u_{,i}$ the i-th composant of a vector $u \in \R^N$. We deduce that for $h<H$, we get 
$$ \abs{v_{h,i}} \leq \abs{v_{h,i} - v_{,i}} + \abs{v_{,i}} \leq  \epsilon + \norm{v_{,i}}_{L^{\infty}(\Omega_T)}.$$
Thus the uniform bound of $v_h$ in $L^{\infty}(\Omega_T)$. 
\end{proof}

\section{Discrete concentration problem}
We recall the problem
\begin{eqnarray}
\label{eqetatbis}
R \psi \partial_t c + \divg (v_c) = -r(c)  + p\chi_{S} \text{ in $\Omega_T$,} 
\end{eqnarray}
with
\begin{eqnarray}
\label{vc}
v_c=-S(v) \psi \nabla c + v c,
\end{eqnarray}
completed by the initial and boundary conditions: 
\begin{eqnarray}
\label{condetatbis}
\chi S(v) \nabla c \cdot n +(1-\chi)c = 0 \text{ on $\partial \Omega \times (0,T)$, $c{|_{t=0}}=c_0$ in $\Omega$.} 
\end{eqnarray}
\begin{remark}
\label{rem:pasdephn}
During the computing of the concentration, the quantity of spread fertilizer is known for all $x$ and for all $t$, $p$ is not a unknown of the problem, that is why it is not discretized. 
\end{remark}
A natural (and classical) scheme would be 
\begin{eqnarray*}
\langle R \psi (c_h^n - c_h^{n-1}), w  \rangle 
+ \tau \langle  \divg({v_c}_h^n),  w   \rangle 
+ \tau \langle r(c_h^{n-1}), w \rangle 
- \tau \langle p , w  \rangle = 0,
\end{eqnarray*}
\begin{eqnarray*}
\langle S(v_h)^{-1} \psi^{-1} {v_c}_h^n , u  \rangle - \langle  c_h^n , \divg(u)  \rangle - \langle S(v_h)^{-1} \psi^{-1} v_h c_h^n , u  \rangle = 0.
\end{eqnarray*}
However, we will adopted the following mixed formulation: 
\begin{eqnarray*}
\langle R \psi (c_h^n - c_h^{n-1}), w  \rangle 
+ \tau \langle  \divg(S(v_h)^{\frac{1}{2}}{v_c}_h^n),  w   \rangle 
+ \tau \langle r(c_h^{n-1}), w \rangle 
- \tau \langle p , w  \rangle = 0,
\end{eqnarray*}
\begin{eqnarray*}
\langle S(v_h)^{-\frac{1}{2}} \psi^{-1} {v_c}_h^n , u  \rangle - \langle  c_h^n , \divg(u)  \rangle - \langle S(v_h)^{-1} \psi^{-1} v_h c_h^n , u  \rangle = 0,
\end{eqnarray*}
for all $w \in W_h$ defined by \eqref{Wh} and for all $u\in U_h$ in the case where $\chi =0$, or for all $u \in U_{h,0}$ else. Indeed, this formulation ensures the stability for $L^2$ velocity (see below). \\
We have denoted by $S(v)^{\frac{1}{2}}$ the unique symmetric matrix square of the symmetric defined positive matrix $S(v_h)$ and we have set $S(v)^{-\frac{1}{2}}=\Bigl(S(v_h)^{\frac{1}{2}}\Bigr)^{-1}$. Thus, we are in a comfortable framework where matrix $S(v)$, $S(v)^{-1}$ and $S(v)^{-\frac{1}{2}}$ are commuting. \\
Thus, the variational problem is the following: \\
\\
\noindent\fbox{\parbox{\linewidth \fboxrule \fboxsep}{
Let $n \in \{ 1,...,N \}$ and $(c_h^{n-1},v_h)$ given in $(W_h,U_h)$ and $p$ given in $L^2(\Omega)$. The problem is to find $c_h^n \in W_h$ such that: 
\begin{multline}
\label{discretc}
\langle R \psi (c_h^n - c_h^{n-1}), w_{h}  \rangle 
+ \tau \langle  \divg(S(v_h)^{\frac{1}{2}}{v_c}_h^n),  w_{h}    \rangle 
+ \tau \langle r(c_h^{n-1}),w_h  \rangle 
- \tau \langle p  ,w_h   \rangle = 0
\end{multline}
\begin{eqnarray}
\label{discretvc}
\langle S(v_h)^{-\frac{1}{2}} \psi^{-1} {v_c}_h^n , u_{h}   \rangle - \langle  c_h^n , \divg(u_h)   \rangle - \langle S(v_h)^{-1} \psi^{-1} v_h c_h^n , u_{h}   \rangle = 0\end{eqnarray}
for all $(w_{h}, u_h) \in (W_h, U_{h})$ such that $\chi u_h \in U_{h,0}$.  
The initialization $c_h^0$ is given in $W_h$. }}\\
\\
\begin{thm}
\label{thm:existeuniquediscretetat}
We suppose that all the assumptions in section  \ref{subsec:assumptions} are satisfied. Then: \\
(i) If $\tau$ is enough small and if $v_h$ is bounded in $L^2(\Omega)$ independently of $h$, problem \eqref{discretc}-\eqref{discretvc} admits a unique solution. \\
(ii) If moreover one of the two following assumptions is satisfied, 
\begin{enumerate}
\item the mesh is assumed quasi-uniform when the diffusion tensor $S(.)$ depends of the velocity, and there exists $\epsilon >0$ such that the following Courant-Friedrich-Lewy condition is satisfied
\begin{eqnarray}
\label{numCFL}
\frac{\tau^{1-\epsilon}}{h^N} := C_{CFL} =  \mathcal{O}(1) \text{ ;} 
\end{eqnarray}
\item  the velocity $v_h$ is bounded in $L^{\infty}(\Omega_T)$ independament of $h$ ; 
\end{enumerate}
then we get the following estimation: 
\begin{eqnarray}
\label{CFLc}
 \norm{c_h^n}_{L^2(\Omega_T)}^2 + \tau \norm{{v_c}_h^n}_{L^2(\Omega_T)}^2 \leq  C.
 \end{eqnarray}
\end{thm}
\begin{remark}
\label{rem:existeuniquediscretetat}
Noticing that the discret velocity $v_h$ is given by the scheme described in subsection \ref{subsec:hydraulic}, the result (ii) is satisfied under assumptions of (ii) in Proposition \ref{prop:convBernardi} (because assumption (b) is then satisfied). The appeal is that the scheme is  well-defined  (perhaps stable: see below) even when the velocity is not controlled $L^{\infty}(\Omega_T)$. 
\end{remark}
\begin{proof}
\textbf{Existence of $c_h^n$}\\
Let
\begin{eqnarray}
\label{base}
\{ w_{1},...,w_{N_1} \} \cup \{ u_{1},...,  u_{N_2} \} \text{ an orthonormality of $W_h \times U_h$}.
\end{eqnarray}
Let $\hat{\alpha}=(\hat{\alpha}_1, ... , \hat{\alpha}_{N_1}) \in \R^{N_1}$, $\alpha=(\alpha_1, ... , \alpha_{N_1}) \in \R^{N_1}$ and $\hat{\beta}=(\hat{\beta}_1, ... , \hat{\beta}_{N_2}) \in \R^{N_2}$, $\beta=(\beta_1, ... , \beta_{N_2}) \in \R^{N_2}$. We denoted $\hat{\xi} = (\hat{\alpha},\hat{\beta}) \in \R^{N_1 + N_2}$ and $\xi = (\alpha,\beta) \in \R^{N_1 + N_2}$. We consider the inner product and the norm on $\R^{N_1 + N_2}$ defined by: 
\begin{eqnarray}
\label{xichapeau}
((\hat{\xi},\xi)):= (\hat{\alpha}, \alpha)_{N_1} + \tau (\hat{\beta},\beta)_{N_2} =  \sum \limits_{i=1}^{N_1} \hat{\alpha}_i \alpha_i + \tau \sum \limits_{i=1}^{N_2} \hat{\beta}_i \beta_i,
\end{eqnarray}
and
\begin{eqnarray}
\label{normexi}
\vert \vert \vert \xi \vert \vert \vert= ((\xi,\xi))^{\frac{1}{2}},
\end{eqnarray}
where $(.,.)_p$ is the euclidean inner product in $\R^p$. Moreover, for all $\xi = (\alpha, \beta) \in \R^{N_1 + N_2}$, there exists a unique $(\bar{w},\bar{u}) \in W_h \times U_h$ defined by 
\begin{eqnarray}
\label{wubar}
\bar{w}=\sum \limits_{k=1}^{N_1} \alpha_k w_k \text{ , } \bar{u}=\sum \limits_{k=1}^{N_2} \beta_k u_k.
\end{eqnarray}
Moreover we have
\begin{align}
\label{wubarxi}
\nonumber \norm{\bar{w}}^2_{L^2(\Omega)} + \tau \norm{\bar{u}}^2_{L^2(\Omega)}  &= \sum \limits_{k=1}^{N_1} \alpha_k^2 \norm{w_k}^2_{L^2(\Omega)} + \tau \sum \limits_{k=1}^{N_2} \beta_k^2 \norm{u_k}^2_{L^2(\Omega)}&& \\
\nonumber &= \sum \limits_{k=1}^{N_1} \alpha_k^2 + \tau \sum \limits_{k=1}^{N_2} \beta_k^2 \text{ car $\norm{w_k}^2_{L^2(\Omega)} = \norm{u_k}^2_{L^2(\Omega)}$ }&&\\
&=1 \text{ according to \eqref{base} } &&\\
&= (\alpha, \alpha)_{N_1} + \tau (\beta, \beta)_{N_2} = ((\xi,\xi)) = \vert \vert \vert \xi \vert \vert \vert^2.&&
\end{align}
For a given $\xi$ and for $(\bar{w}, \bar{u})$ defined by \eqref{wubar}, we use equations \eqref{discretc} and \eqref{discretvc} in order to define componants of $\hat{\xi}= (\hat{\alpha},\hat{\beta}) \in \R^{N_1 + N_2}$ as follows: 
\begin{eqnarray}
\label{alphachapeau}
\hat{\alpha}_k = \langle R \psi (\bar{w} - c_h^{n-1}), w_{k}  \rangle 
+ \tau \langle  \divg (S(v_h)^{\frac{1}{2}} \bar{u}),  w_{k}    \rangle 
+ \tau \langle r(c_h^{n-1}),w_{k}  \rangle 
- \tau \langle p, w_{k}   \rangle = 0
\end{eqnarray}
for all $k=1,...,N_1$ and
\begin{eqnarray}
\label{betachapeau}
\hat{\beta}_k = \langle S(v_h)^{-\frac{1}{2}} \psi^{-1} \bar{u} , u_{k}   \rangle - \langle  \bar{w} , \divg(u_{k})   \rangle - \langle S(v_h)^{-1} \psi^{-1} v_h \bar{w} , u_{k}   \rangle = 0
\end{eqnarray}
for all $k=1,...,N_2$. \\
Finally, we define the application $\mathcal{P} : \R^{N_1 + N_2} \rightarrow \R^{N_1 + N_2}$ by $$\mathcal{P}(\xi)=\hat{\xi}$$
which is continuous by construction. \\
The problem is then to solve $$ \mathcal{P}(\xi)=0.$$
We build $((\mathcal{P}(\xi),\xi))=((\hat{\xi},\xi))$ as defined in \eqref{xichapeau}. We compute the following euclidian inner product in $\R^{N_1}$: 
\begin{multline}
\label{psalpha}
(\hat{\alpha},\alpha)_{N_1} = \langle R \psi (\bar{w} - c_h^{n-1}), \bar{w } \rangle 
+ \tau \langle  \divg (S(v_h)^{\frac{1}{2}}\bar{u}),  \bar{w}   \rangle 
+ \tau \langle r(c_h^{n-1}), \bar{w} \rangle 
- \tau \langle p, \bar{w}   \rangle,
\end{multline}
and then the euclidian inner product in $\R^{N_2}$ : 
\begin{multline}
\label{psbeta}
\tau (\hat{\beta},\beta)_{N_2} = \tau \langle S(v_h)^{-\frac{1}{2}} \psi^{-1} \bar{u} ,S(v_h)^{\frac{1}{2}} \bar{u}   \rangle - \tau \langle  \bar{w} , \divg (S(v_h)^{\frac{1}{2}} \bar{u} )  \rangle \\ - \tau \langle S(v_h)^{-1} \psi^{-1} v_h \bar{w} , S(v_h)^{\frac{1}{2}} \bar{u}  \rangle.
 \end{multline}
 By adding \eqref{psalpha} and \eqref{psbeta}, we get: 
 \begin{multline}
 \label{pspxi}
 ((\mathcal{P}(\xi),\xi))=  \langle R \psi \bar{w} , \bar{w } \rangle  - \langle R \psi c_h^{n-1}, \bar{w } \rangle 
+ \tau \langle r(c_h^{n-1}), \bar{w} \rangle 
- \tau \langle p, \bar{w}   \rangle \\
+  \tau \langle S(v_h)^{-\frac{1}{2}} \psi^{-1} \bar{u} , S(v_h)^{\frac{1}{2}} \bar{u} \rangle - \tau \langle S(v_h)^{-1} \psi^{-1} v_h \bar{w} , S(v_h)^{\frac{1}{2}} \bar{u}  \rangle.
 \end{multline}
Now, we will prove that $(\mathcal{P}(\xi),\xi) > 0$. To reach this, we will estimate terms of \eqref{pspxi}, in particular through Cauchy-Schwarz and Young inequalities. We get: 
$$I_1=  \langle R \psi \bar{w} ,\bar{w}  \rangle \geq R \psi_-  \norm{\bar{w}}^2_{L^2(\Omega)},$$
$$\abs{I_2}=  \abs{\langle R \psi c_h^{n-1}, \bar{w}  \rangle } \leq  \frac{R \psi_-}{8} \norm{\bar{w}}^2_{L^2(\Omega)} + \frac{2R\psi_+^2}{\psi_-} \norm{c_h^{n-1}}^2_{L^2(\Omega)}, $$
$$\abs{I_3}= \abs{\tau \langle r(c_h^{n-1}), \bar{w} \rangle} \leq \frac{R \psi_-}{8} \norm{\bar{w}}^2_{L^2(\Omega)} + \frac{2\tau^2 r_+}{R\psi_-}  \norm{c_h^{n-1}}^2_{L^2(\Omega)} , $$ 
$$\abs{I_4}= \abs{ \tau \langle p_h^n, \bar{w} \rangle} \leq  \frac{R \psi_-}{8} \norm{\bar{w}}^2_{L^2(\Omega)} + \frac{2\tau^2}{R\psi_-} \norm{p}^2_{L^2(\Omega)}, $$
$$I_5= \tau \langle S(v_h)^{-\frac{1}{2}} \psi^{-1} \bar{u}, S(v_h^n)^{\frac{1}{2}}\bar{u}  \rangle \geq  \tau \psi^{-1}_+ \norm{\bar{u}}^2_{L^2(\Omega)}, $$
\begin{eqnarray*}
\abs{I_6} = \abs{ \tau \langle S(v_h)^{-1} \psi^{-1} v_h \bar{w}, S(v_h)^{\frac{1}{2}} \bar{u}  \rangle }. 
\end{eqnarray*}
Using equation \eqref{eq5.7.2}, we get
\begin{eqnarray*}
\abs{I_6} \leq \tau \psi^{-1} C_{disp} \Bigl(1+ \norm{v_h}_{L^{\infty}(\Omega)}^{\frac{1}{2}} \Bigr) \norm{\bar{w}}_{L^2(\Omega)} \norm{\bar{u}}_{L^2(\Omega)}.
\end{eqnarray*} 
With Cauchy-Schwarz and Young inequalities, we get
\begin{eqnarray*}
\abs{I_6} \leq \frac{R \psi_-}{8}  \norm{\bar{w}}_{L^{2}(\Omega)}^2 +  \frac{ 2 \tau^2  C_{disp}^2 \Bigl(1+ \norm{v_h}_{L^{\infty}(\Omega)}^{\frac{1}{2}} \Bigr)^2}{R \psi^3_-} \norm{\bar{u}}_{L^{2}(\Omega)}^2  .
\end{eqnarray*} 
Thus: 
\begin{eqnarray*}
((\mathcal{P}(\xi),\xi)) \geq \frac{R \psi_-}{2}  \norm{\bar{w}}^2_{L^2(\Omega)} + \tau \Bigl( \psi^{-1}_+ -C' \Bigr) \norm{\bar{u}}^2_{L^2(\Omega)} - K
\end{eqnarray*}
where $C'= \frac{2\tau  C_{disp}^2 \Bigl(1+ \norm{v_h}_{L^{\infty}(\Omega)}^{\frac{1}{2}} \Bigr)^2}{ R \psi^3_-}$ and\\
$K=  \frac{2}{\psi_-} \Bigl((R \psi_+^2 + \frac{\tau^2 r_+}{R}) \norm{c_h^{n-1}}^2_{L^2(\Omega)} +\frac{\tau^2}{R}  \norm{p}^2_{L^2(\Omega)}\Bigr)$. \vspace{3mm}\\ 
By taking a time step small enough such that $\psi^{-1}_+ - C'  > 0$ and by denoted $\displaystyle m= \min \{ \frac{R\psi_-}{2}, \psi^{-1}_+ - C' \}$, we get
$$((P(\xi),\xi)) \geq m \vert \vert \vert \xi \vert \vert \vert^2 - K.$$
We are now able to prove that for all $\xi \in \R^{N_1+N_2}$ satisfying $\vert \vert \vert \xi \vert \vert \vert^2 = 2 K /m$, we get $((P(\xi),\xi)) \geq K > 0$. According to Lemme 1.4 page 164 of Temam \cite{Temam}, there exists a $\xi \in \R^{N_1 + N_2}$ such that $P(\xi)=0$. We deduce the existence of $c_h^{n}$ solution of the problem. \\
\\
\textbf{Uniqueness of $c_h^n$}\\
The discretized space step $h$ being unchanged, we will omit it thereafter. We assume that there exist two solutions $(c_{1}^{n},v_{c,1}^n)$,$(c_{2}^{n},v_{c,2}^n)$ $\in W_h^2 \times U_h^2$ to \eqref{discretc} from the same solution $(c_{h}^{n-1},{v_{c}}_h^{n-1})$. They thus satisfied
\begin{align}
\label{discretc1}
\langle R \psi (c_{1}^n - c_h^{n-1}), w_{h}  \rangle 
+ \tau \langle  \divg (S(v_h)^{\frac{1}{2}}v_{c,1}^n),  w_{h}    \rangle 
+ \tau \langle r(c_h^{n-1}),w_h  \rangle 
- \tau \langle p,w_h   \rangle = 0 \\
\langle S(v_h)^{-\frac{1}{2}} \psi^{-1} v_{c,1}^n , u_{h}   \rangle - \langle  c_1^n , \divg(u_h)   \rangle - \langle S(v_h)^{-1} \psi^{-1} v_h c_1^n , u_{h}   \rangle = 0
\label{discretc1b}
\end{align}
and
\begin{align}
\label{discretc2}
\langle R \psi (c_2^n - c_h^{n-1}), w_{h}  \rangle 
+ \tau \langle  \divg (S(v_h)^{\frac{1}{2}}v_{c,2}^n),  w_{h}    \rangle 
+ \tau \langle r(c_h^{n-1}),w_h  \rangle 
- \tau \langle p,w_h   \rangle = 0 \\
\langle S(v_h)^{-\frac{1}{2}} \psi^{-1} v_{c,2}^n , u_{h}   \rangle - \langle  c_2^n , \divg(u_h)   \rangle - \langle S(v_h)^{-1} \psi^{-1} v_h c_2^n , u_{h}   \rangle = 0.
\label{discretc2b}
\end{align}
We substract \eqref{discretc1} and \eqref{discretc2}, and \eqref{discretc1b} and \eqref{discretc2b}. Let $C=c_1^{n} - c_2^{n}$ and $V_c =  v_{c,1}^n - v_{c,2}^n$. We get
\begin{align}
\label{discretcbis}
\langle R \psi C, w_{h}  \rangle 
+ \tau \langle  \divg (S(v_h)^{\frac{1}{2}}V_{c}),  w_{h}    \rangle  = 0 \\
\langle S(v_h)^{-\frac{1}{2}} \psi^{-1} V_{c} , u_{h}   \rangle - \langle  C , \divg(u_h)   \rangle - \langle S(v_h)^{-1} \psi^{-1} v_h C , u_{h}   \rangle = 0.
\label{discretcbis2}
\end{align}
In order to remove the second terms in each equations, we choose the test functions $w_h = C$ and $u_h= \tau S(v_h)^{\frac{1}{2}} V_c$ and we add \eqref{discretcbis} and \eqref{discretcbis2}. We get: 
 \begin{eqnarray}
 \label{discretC}
\langle R \psi C , C \rangle + \tau \langle  S(v_h)^{-\frac{1}{2}} \psi^{-1}  V_c, S(v_h)^{\frac{1}{2}}V_c \rangle   - \tau \langle S(v_h)^{-1} \psi^{-1} v_h C, S(v_h)^{\frac{1}{2}}V_c \rangle = 0.
\end{eqnarray}
We now estimate the three terms of the equations as follows: 
$$I_1 = \langle R \psi C , C\rangle \geq  R \psi_- \norm{C}^2_{L^2(\Omega)}, $$ 
$$I_2= \tau \langle  S(v_h)^{-\frac{1}{2}} \psi^{-1}  V_c, S(v_h)^{\frac{1}{2}}V_c \rangle  \geq \tau \psi^{-1}_+  \norm{V_c}^2_{L^2(\Omega)}.$$
According to \eqref{eq5.7.2},
$$\abs{I_3}= \abs{\tau \langle S(v_h)^{-1} \psi^{-1} v_h C, S(v_h)^{\frac{1}{2}}V_c \rangle} \leq \tau \psi_{-}^{-1}  C_{disp} \Bigl(1+ \norm{v_h}_{L^{\infty}(\Omega)}^{\frac{1}{2}} \Bigr) \norm{C}_{L^2(\Omega)}  \norm{V_c}_{L^2(\Omega)} $$
thus
\begin{align*}\abs{I_3}= \abs{\tau \langle S(v_h)^{-1} \psi^{-1} v_h C, S(v_h)^{\frac{1}{2}}V_c \rangle} \leq & \frac{R \psi_{-}}{2} \norm{C}_{L^{2}(\Omega)}^2 &&\\
&+ \frac{ \tau^2  C_{disp}^2 \Bigl(1+ \norm{v_h}_{L^{\infty}(\Omega)}^{\frac{1}{2}} \Bigr)^2}{2 R \psi^3_-} \norm{V_c}_{L^2(\Omega)}^2.  &&
\end{align*}
Equation \eqref{discretC} is then
$$ \frac{R \psi_-}{2}  \norm{C}_{L^2(\Omega)}^2  + \tau \Bigl(\psi^{-1}_+ - C'' \Bigr) \norm{V_c}_{L^2(\Omega)}^2\leq 0$$
with $C''= \frac{ \tau  C_{disp}^2 \Bigl(1+ \norm{v_h}_{L^{\infty}(\Omega)}^{\frac{1}{2}} \Bigr)^2}{2 R \psi^3_-} $. \\
For a time step enough small such that $\psi_+^{-1} - C''  > 0$, we get from the previews inequality: 
$$ \norm{C}_{L^2(\Omega)}^2= 0 \text{ and } \norm{V_c}_{L^2(\Omega)}^2=0 \Rightarrow c_1^n = c_{2}^{n} \text{ and } v_{c,1}^n = v_{c,2}^n = 0 \text{ almost everywhere in $\Omega$.}$$
The solution $(c_h^n, {v_c}_h^n)$ of the problem \eqref{discretc}-\eqref{discretvc} is then unique. \\
\\
\textbf{Estimation in function of $h$ and $\tau$} \label{page:traitementI6} \\
\\
In order to obtain estimations in function of $h$ and $\tau$, we use the estimations of $I_6$ : \label{pageI6}
\begin{eqnarray*}
\abs{I_6} \leq \tau  \langle \psi^{-1}  C_{disp} \Bigl(1+ \norm{v_h}_{L^{\infty}(\Omega)}^{\frac{1}{2}} \Bigr) \bar{w}, \bar{u}  \rangle  \leq \tau \psi^{-1}_- C_{disp} \sum \limits_{\mathcal{O}_h \in \mathcal{T}_h} \int_{\mathcal{O}_h} (1+\abs{v_h}^{\frac{1}{2}}) \abs{\bar{w}_h} \abs{\bar{u}_h} \dx .
\end{eqnarray*} 
In the case where assumption (ii) (b) of Theorem \ref{thm:existeuniquediscretetat} is satisfied, estimation of this term is obvious and we directly get \eqref{CFLc}. Else, by the Cauchy-Schwarz inequality, we get
\begin{align*}
\abs{I_6} &\leq \tau \psi^{-1}_- C_{disp} \sum \limits_{\mathcal{O}_h \in \mathcal{T}_h} \Bigl( \int_{\mathcal{O}_h}(1+ \abs{v_h}^{\frac{1}{2}})^2 \dx \Bigr)^{\frac{1}{2}} \Bigl(  \int_{\mathcal{O}_h} \abs{\bar{w}_h}^2\abs{\bar{u}_h}^2 \dx \Bigr)^{\frac{1}{2}}&& \\
&\leq \tau \psi^{-1}_- C_{disp} \sum \limits_{\mathcal{O}_h \in \mathcal{T}_h} \Bigl( \int_{\mathcal{O}_h} C \abs{v_h} \dx \Bigr)^{\frac{1}{2}} \Bigl(  \int_{\mathcal{O}_h} \abs{\bar{w}_h}^2\abs{\bar{u}_h}^2 \dx \Bigr)^{\frac{1}{2}}&&\\
& \leq C \tau \psi^{-1}_- C_{disp}  \sum \limits_{\mathcal{O}_h \in \mathcal{T}_h}   \Biggl( \Bigl(\int_{\mathcal{O}_h}  \abs{v_h}^2  \dx \Bigr)^{\frac{1}{2}} \Bigl(\int_{\mathcal{O}_h}   1^2 \dx  \Bigr)^{\frac{1}{2}} \Biggr)^{\frac{1}{2}}  \Biggl(  \int_{\mathcal{O}_h} \abs{\bar{w}_h}^2\abs{\bar{u}_h}^2 \dx \Biggr)^{\frac{1}{2}}.  &&
\end{align*}
Using that $\norm{v_h}_{L^2(\mathcal{O}_h)} \leq \norm{v_h}_{L^2(\Omega)} $, quantity that we suppose uniform boundedness in Theorem \ref{thm:existeuniquediscretetat}, we find 
\begin{align*}
\abs{I_6} &\leq C \tau \psi^{-1}_- C_{disp}  \sum \limits_{\mathcal{O}_h \in \mathcal{T}_h}  \norm{v_h}^{\frac{1}{2}}_{L^2(\Omega)} \abs{\mathcal{O}_h}^{\frac{1}{4}} \Bigl(  \int_{\mathcal{O}_h} \abs{\bar{w}_h}^2\abs{\bar{u}_h}^2 \dx \Bigr)^{\frac{1}{2}}  &&\\
&\leq C \tau \psi^{-1}_- C_{disp}   \sum \limits_{\mathcal{O}_h \in \mathcal{T}_h} \abs{\mathcal{O}_h}^{\frac{1}{4}} \Bigl(  \int_{\mathcal{O}_h} \abs{\bar{w}_h}^2\abs{\bar{u}_h}^2 \dx \Bigr)^{\frac{1}{2}}.&&
\end{align*}
As $\bar{w}_h$ is constant by mesh, we get
\begin{align*}
\abs{I_6} &\leq C\tau \psi^{-1}_- C_{disp}  \sum \limits_{\mathcal{O}_h \in \mathcal{T}_h} \abs{\mathcal{O}_h}^{\frac{1}{4}} \abs{\bar{w}_h} \Bigl(  \int_{\mathcal{O}_h}  \abs{\bar{u}_h}^2 \dx \Bigr)^{\frac{1}{2}} &&\\
&= C\tau \psi^{-1}_- C_{disp}  \sum \limits_{\mathcal{O}_h \in \mathcal{T}_h} \abs{\mathcal{O}_h}^{\frac{1}{4}}   \Bigl(  \frac{1}{\abs{\mathcal{O}_h}} \int_{\mathcal{O}_h}  \abs{\bar{w}_h}^2 \dx \Bigr)^{\frac{1}{2}}  \Bigl(  \int_{\mathcal{O}_h}  \abs{\bar{u}_h}^2 \dx \Bigr)^{\frac{1}{2}}. &&
\end{align*}
Supposing the mesh is quasi-uniform (because we use finit elements $P_0$ constant by mesh, in other cases we have to adapted the proof by using inverse inequalities (see theorem 4.5.11 and the associated remark in Brenner and Scott  \cite{Brenner})), we get
\begin{eqnarray*}
\abs{I_6} \leq 
 \frac{C\tau \psi^{-1}_- C_{disp}}{h^{\frac{N}{4}}} \sum \limits_{\mathcal{O}_h \in \mathcal{T}_h} \Bigl( \norm{\bar{w}_h}_{L^2(\mathcal{O}_h)} \norm{\bar{u}_h}_{L^2(\mathcal{O}_h)}   \Bigr). 
  \end{eqnarray*}
By the discret Cauchy-Schwarz inequality, we get 
 \begin{align*}
\abs{I_6} &\leq 
 \frac{C\tau \psi^{-1}_- C_{disp}}{h^{\frac{N}{4}}}  \Bigl( \sum \limits_{\mathcal{O}_h \in \mathcal{T}_h}  \norm{\bar{w}_h}^{2}_{L^2(\mathcal{O}_h)}  \Bigr)^{\frac{1}{2}} \Bigl( \sum \limits_{\mathcal{O}_h \in \mathcal{T}_h}  \norm{\bar{u}_h}^{2}_{L^2(\mathcal{O}_h)}  \Bigr)^{\frac{1}{2}}  &&\\
&\leq \frac{C\tau \psi^{-1}_- C_{disp}}{h^{\frac{N}{4}}}  \norm{\bar{w}}_{L^2(\Omega)}   \norm{\bar{u}}_{L^2(\Omega)}. &&
\end{align*}
Finally, according to Young inequality, we get 
\begin{eqnarray*}
\abs{I_6} \leq 
 \frac{R \psi_-}{8} \norm{\bar{w}}_{L^2(\Omega)}^2  +  \frac{2C^2 \tau^2 C_{disp}^2}{h^{\frac{N}{2}}R \psi^3_- }  \norm{\bar{u}}_{L^2(\Omega)}^2 
= \frac{R \psi_-}{8} \norm{\bar{w}}_{L^2(\Omega)}^2  +  C'  \tau \norm{\bar{u}}_{L^2(\Omega)}^2
\end{eqnarray*}
where $C'=  \frac{2C^2 \tau C_{disp}^2}{h^{\frac{N}{2}}R \psi^3_- } = \frac{2C^2 C_{disp}^2 \tau^{\frac{1-\epsilon}{2}} \tau^{\frac{1+\epsilon}{2}}}{h^{\frac{N}{2}}R \psi^3_-} $. \\
As we have proved the existence and the uniqueness of $c_h^n$ and ${v_c}_h^n$, we can write: 
\begin{eqnarray*}
((\mathcal{P}(\xi),\xi)) \geq \frac{R \psi_-}{2}  \norm{c_h^n}^2_{L^2(\Omega)} + \tau \Bigl( \psi^{-1}_+ - C'  \Bigr) \norm{{v_c}_h^n}^2_{L^2(\Omega)} - K
\end{eqnarray*}
where $K=  \frac{2}{\psi_-} \Bigl((R \psi_+^2 + \frac{\tau^2 r_+}{R}) \norm{c_h^{n-1}}^2_{L^2(\Omega)} + \frac{\tau^2}{R}  \norm{p}^2_{L^2(\Omega)}\Bigr)$. \vspace{3mm}\\
If the Courant-Friedrichs-Lewy condition \eqref{numCFL} is satisfied, 
then we get
\begin{eqnarray*}
((\mathcal{P}(\xi),\xi)) \geq \frac{R \psi_-}{2}  \norm{c_h^n}^2_{L^2(\Omega)} + \tau \Bigl( \psi^{-1}_+ - C_{CFL} \tau^{\frac{1+\epsilon}{2}}  \Bigr) \norm{{v_c}_h^n}^2_{L^2(\Omega)} - K.
\end{eqnarray*}
As the computing of $((P(\xi),\xi))$ now implies uniform constants in $h$ and $\tau$, we deduce uniform estimations of the theorem for $\tau$ small enough (such that $\psi_+^{-1} - C_{\footnotesize{CFL}} \tau^{1 + \frac{\epsilon}{2}} > 0$).
\end{proof} 

\textbf{Stability} \\
Lets now focus on the study of the stability. We state the following result:
\begin{thm}
\label{thm:stabilite}
Assume that $(c_h^n,{v_c}_h^n)$ is solution of problem \eqref{discretc}-\eqref{discretvc}. Assume that $r \in L^{\infty}(\R)$ or $\abs{r(x)} \leq r^+ \abs{x}$ for all $x \in \R$ with $r^+ \in \R$ and the condition \eqref{CFLc} is satisfied. We have the following estimations
\begin{multline}
\label{stabilite1}
\frac{1}{\tau}  \sum \limits_{n=1}^{T/ \tau} \norm{c_h^n - c_h^{n-1}}^2_{L^2(\Omega)}  + \sup \limits_{0\leq k \leq T/\tau} \norm{{v_c}_h^k}^2_{L^2(\Omega)} + \sum \limits_{n=1}^{T/ \tau} \norm{{v_c}_h^n - {v_c}_h^{n-1}}^2_{L^2(\Omega)}\\
 \leq C \Bigl(1 +  \norm{v_h}_{L^{\infty}(\Omega)}+ \norm{v_h}_{L^{\infty}(\Omega)}^2 \Bigr)
\end{multline}
and
\begin{eqnarray}
\label{stabilite2}
 \tau \sum \limits_{n=1}^{T/\tau} \norm{ \divg(S(v_h)^{\frac{1}{2}} {v_c}_h^n)}^2_{L^2(\Omega)} \leq C \Bigl(1 +  \norm{v_h}_{L^{\infty}(\Omega)}+ \norm{v_h}_{L^{\infty}(\Omega)}^2 \Bigr). 
\end{eqnarray}
\end{thm}

\begin{proof}
We recall \eqref{discretc} and \eqref{discretvc}:
\begin{eqnarray*}
\langle R \psi (c_h^n - c_h^{n-1}),w_h \rangle + \tau \langle \divg(S(v_h)^{\frac{1}{2}} {v_c}_h^n),w_h \rangle + \tau \langle  r(c_h^{n-1}),w_h \rangle - \tau \langle p,w_h \rangle = 0, \\
 \langle S(v_h)^{-\frac{1}{2}} \psi^{-1} {v_c}_h^n ,u_h \rangle - \langle  c_h^n,\divg(u_h) \rangle -  \langle S(v_h)^{-1} \psi^{-1} v_h c_h^n,u_h \rangle = 0.
\end{eqnarray*}
By rewriting equation \eqref{discretvc} in time $t^n$ and $t^{n-1}$ and by substracting the two equations, we get: 
\begin{align}
\label{nouveaudiscretvc}
 \langle S(v_h)^{-\frac{1}{2}} \psi^{-1} ({v_c}_h^n - {v_c}_h^{n-1}),u_h \rangle &- \langle  (c_h^n - c_h^{n-1}),\divg(u_h) \rangle  \nonumber &&\\
 &-  \langle S(v_h)^{-1} \psi^{-1} v_h (c_h^n - c_h^{n-1}),u_h \rangle = 0.&&
\end{align}
Adding equations \eqref{discretc} and \eqref{nouveaudiscretvc} and letting $w_h = c_h^n - c_h^{n-1}$ and $u_h = \tau S(v_h)^{\frac{1}{2}}  {v_c}_h^n$ : 
\begin{multline*}
\langle R \psi (c_h^n - c_h^{n-1}),(c_h^n - c_h^{n-1}) \rangle + \tau \langle  r(c_h^{n-1}),(c_h^n - c_h^{n-1}) \rangle - \tau \langle p,(c_h^n - c_h^{n-1}) \rangle \\
 + \tau \langle S(v_h)^{-\frac{1}{2}} \psi^{-1} ({v_c}_h^n - {v_c}_h^{n-1}),S(v_h)^{\frac{1}{2}}  {v_c}_h^n \rangle  - \tau \langle S(v_h)^{-1} \psi^{-1} v_h (c_h^n - c_h^{n-1}),S(v_h)^{\frac{1}{2}}  {v_c}_h^n \rangle= 0.
\end{multline*}
We estimated the five terms: 
\begin{align*}
I_1 &\geq R \psi_- \norm{c_h^n - c_h^{n-1}}^2_{L^2(\Omega)}, &&\\
 \abs{I_2} &\leq \frac{\tau^2}{R \psi_-}  \norm{r(c_h^{n-1})}^2_{L^2(\Omega)} + \frac{R \psi_-}{4} \norm{c_h^n - c_h^{n-1}}^2_{L^2(\Omega)}, &&\\
 \abs{I_3} &\leq \frac{\tau^2}{R \psi_-}  \norm{p}^2_{L^2(\Omega)} + \frac{R \psi_-}{4} \norm{c_h^n - c_h^{n-1}}^2_{L^2(\Omega)}, &&\\
 I_4 &= \frac{\tau}{2} \int_{\Omega} \psi^{-1} \Bigl( ({v_c}_h^n)^2 - ({v_c}_h^{n-1})^2 + ({v_c}_h^n - {v_c}_h^{n-1})^2 \Bigr) \dx &&\\
&= \frac{\tau}{2} \norm{\psi^{-\frac{1}{2}} {v_c}_h^n}^2_{L^2(\Omega)} + \frac{\tau}{2} \norm{\psi^{-\frac{1}{2}} ({v_c}_h^n - {v_c}_h^{n-1})}^2_{L^2(\Omega)} - \frac{\tau}{2} \norm{\psi^{-\frac{1}{2}} {v_c}_h^{n-1}}^2_{L^2(\Omega)},&&
\end{align*}
because we have assumed $S(v_h)$ is symmetric, and by using \eqref{eq5.7.2} 
\begin{align*} 
\abs{I_5} &\leq  \frac{\tau^2 C_{disp}^2 \Bigl(1+ \norm{v_h}_{L^{\infty}(\Omega)}^{\frac{1}{2}} \Bigr)^2}{R \psi^3_-} \norm{{v_c}_h^n}^2_{L^2(\Omega)} + \frac{R \psi_-}{4} \norm{ c_h^n - c_h^{n-1}}^2_{L^2(\Omega)} &&\\
 &\leq  \frac{C \tau^2 C_{disp}^2 \norm{v_h}_{L^{\infty}(\Omega)}}{R \psi^3_-} \norm{{v_c}_h^n}^2_{L^2(\Omega)} + \frac{R \psi_-}{4} \norm{ c_h^n - c_h^{n-1}}^2_{L^2(\Omega)}.&&
 \end{align*}
We thus get
\begin{align*}
\frac{R\psi_-}{4} \norm{ c_h^n - c_h^{n-1}}^2_{L^2(\Omega)}& + \frac{\tau}{2}  \norm{\psi^{-\frac{1}{2}} ({v_c}_h^n - {v_c}_h^{n-1})}^2_{L^2(\Omega)} + \frac{\tau}{2} \norm{\psi^{-\frac{1}{2}} {v_c}_h^n}^2_{L^2(\Omega)} &&\\
&\leq \frac{\tau}{2} \norm{\psi^{-\frac{1}{2}} {v_c}_h^{n-1}}^2_{L^2(\Omega)}&& \\
& + \frac{\tau^2}{R \psi_-} \Bigl(  \norm{r(c_h^{n-1})}^2_{L^2(\Omega)} + \norm{p}^2_{L^2(\Omega)} &&\\
&+  C_{disp}^2 C\norm{v_h}_{L^{\infty}(\Omega)}\psi^{-2}_- \norm{{v_c}_h^n}^2_{L^2(\Omega)} \Bigr).&&
\end{align*}
If $r \in L^{\infty}(\Omega)$, we get $\norm{r(c_h^{n-1})}^2_{L^2(\Omega)}  \leq C$. Similarly, we have assumed that condition \eqref{CFLc} is satisfied, thus we get $\norm{c_h^n}_{L^2(\Omega)} \leq C$. Then, if $\abs{r(x)} \leq r^+ \abs{x}$, $x \in \R$, we get $\norm{r(c_h^{n-1})}^2_{L^2(\Omega)}  \leq C$. Thus
\begin{multline*}
\frac{R\psi_-}{4} \norm{ c_h^n - c_h^{n-1}}^2_{L^2(\Omega)} + \frac{\tau}{2}  \norm{\psi^{-\frac{1}{2}} ({v_c}_h^n - {v_c}_h^{n-1})}^2_{L^2(\Omega)} + \frac{\tau}{2} \norm{\psi^{-\frac{1}{2}} {v_c}_h^n}^2_{L^2(\Omega)} \\
 \leq \frac{\tau}{2} \norm{\psi^{-\frac{1}{2}} {v_c}_h^{n-1}}^2_{L^2(\Omega)} 
 + \frac{\tau^2}{R \psi_-} \Bigl( C+  C_{disp}^2 C \norm{v_h}_{L^{\infty}(\Omega)} \psi^{-2}_- \norm{{v_c}_h^n}^2_{L^2(\Omega)} \Bigr).
\end{multline*}
By adding for $n=1,...,k$ and noticing that $\sum \limits_{n\leq T/\tau} C \tau^2 \leq C \tau$, we get
\begin{multline}
\label{somme}
C \sum \limits_{n=1}^k \norm{c_h^n - c_h^{n-1}}^2_{L^2(\Omega)} + C \tau \sum \limits_{n=1}^k \norm{{v_c}_h^n - {v_c}_h^{n-1}}^2_{L^2(\Omega)} + \frac{\tau}{2} \norm{{v_c}_h^{k}}^2_{L^2(\Omega)} \\
 \leq C \tau + 2 \tau \norm{v_h}_{L^{\infty}(\Omega)} \Bigl( \sum \limits_{n=1}^k \frac{\tau}{2} \norm{{v_c}_h^n}^2_{L^2(\Omega)} \Bigr).
\end{multline}
Lets start by considering
$$ \tau \norm{{v_c}_h^k}^2_{L^2(\Omega)} \leq C \tau + 2 \tau \norm{v_h}_{L^{\infty}(\Omega)} \Bigl( \sum \limits_{n=1}^k \frac{\tau}{2} \norm{{v_c}_h^n}^2_{L^2(\Omega)} \Bigr). $$
The discret Gronwall inequality \footnote{If $(a_k)_{k \geq 0}$, $(b_k)_{k \geq 0}$, $(c_k)_{k \geq 0}$, are three sequences with positive termes such that for all $k \geq 0$ we get $a_{k+1} \leq c_{k+1} + \sum \limits_{n=0}^{k} a_n b_n$, then $a_{k+1} \leq c_{k+1} + \sum \limits_{n=0}^{k} c_n b_n \exp(\sum \limits_{j=n+1}^{k} b_j)$.} gives : 
$$ \norm{{v_c}_h^n}^2_{L^2(\Omega)} \leq C \Bigl(1 +  \norm{v_h}_{L^{\infty}(\Omega)} \Bigr). $$
We can also deduce from \eqref{somme} : 
$$ \sum \limits_{n=1}^k \norm{c_h^n - c_h^{n-1}}^2_{L^2(\Omega)} \leq C \tau \Bigl(1 +  \norm{v_h}_{L^{\infty}(\Omega)} + \norm{v_h}_{L^{\infty}(\Omega)}^2 \Bigr)$$
and
$$\sum \limits_{n=1}^k \norm{{v_c}_h^n - {v_c}_h^{n-1}}^2_{L^2(\Omega)}   \leq C \Bigl(1 +  \norm{v_h}_{L^{\infty}(\Omega)}+ \norm{v_h}_{L^{\infty}(\Omega)}^2 \Bigr).$$
As the results are satisfied for all $k$, $0 \leq k \leq T/ \tau$, we conclude: 
\begin{multline}
\label{estimc}   
 \frac{1}{\tau}  \sum \limits_{n=1}^{T/ \tau} \norm{c_h^n - c_h^{n-1}}^2_{L^2(\Omega)}  + \sup \limits_{0\leq k \leq T/\tau} \norm{{v_c}_h^k}^2_{L^2(\Omega)} + \sum \limits_{n=1}^{T/ \tau} \norm{{v_c}_h^n - {v_c}_h^{n-1}}^2_{L^2(\Omega)} \\ 
 \leq C \Bigl(1 +  \norm{v_h}_{L^{\infty}(\Omega)} + \norm{v_h}_{L^{\infty}(\Omega)}^2 \Bigr).
\end{multline}
We return to equation \eqref{discretc} by taking $w_h = \divg(S(v_h)^{\frac{1}{2}} {v_c}_h^n)$. We get: 
\begin{multline*}
\langle R \psi (c_h^n - c_h^{n-1}),\divg(S(v_h)^{\frac{1}{2}} {v_c}_h^n) \rangle + \tau \langle \divg(S(v_h)^{\frac{1}{2}} {v_c}_h^n),\divg(S(v_h)^{\frac{1}{2}} {v_c}_h^n) \rangle \\
+ \tau \langle  r(c_h^{n-1}),\divg(S(v_h)^{\frac{1}{2}} {v_c}_h^n) \rangle - \tau \langle p,\divg(S(v_h)^{\frac{1}{2}} {v_c}_h^n) \rangle = 0.
\end{multline*}
We estimate the following terms: 
\begin{align*}
 \abs{I_1} &\leq \frac{R^2 \psi_+^2}{\tau} \norm{c_h^n - c_h^{n-1}}^2_{L^2(\Omega)} + \frac{\tau}{4} \norm{ \divg(S(v_h)^{\frac{1}{2}} {v_c}_h^n)}^2_{L^2(\Omega)},  &&\\
 I_2& \geq \tau \norm{ \divg(S(v_h)^{\frac{1}{2}} {v_c}_h^n)}^2_{L^2(\Omega)}, &&\\
 \abs{I_3} &\leq C\tau + \frac{\tau}{4} \norm{ \divg(S(v_h)^{\frac{1}{2}} {v_c}_h^n)}^2_{L^2(\Omega)}&&
 \end{align*}
because $r$ is assumed $L^{\infty}(\Omega)$ or sublinear with condition \eqref{CFLc}, and
$$ \abs{I_4} \leq C \tau + \frac{\tau}{4} \norm{ \divg(S(v_h)^{\frac{1}{2}} {v_c}_h^n)}^2_{L^2(\Omega)}. $$
Thus we get
\begin{eqnarray*}
\frac{\tau}{4} \norm{ \divg(S(v_h)^{\frac{1}{2}} {v_c}_h^n)}^2_{L^2(\Omega)} \leq C \tau + \frac{R^2 \psi_+^2}{\tau} \norm{ c_h^n - c_h^{n-1}}^2_{L^2(\Omega)}.
\end{eqnarray*}
Adding for $n=1,...,k$, we get 
\begin{eqnarray*}
C \tau \sum \limits_{n=1}^k \norm{ \divg(S(v_h)^{\frac{1}{2}} {v_c}_h^n)}^2_{L^2(\Omega)} \leq C+ \frac{C}{\tau} \sum \limits_{n=1}^k \norm{ c_h^n - c_h^{n-1}}^2_{L^2(\Omega)}.
\end{eqnarray*}
According to \eqref{estimc}, $\frac{C}{\tau} \sum \limits_{n=1}^k \norm{ c_h^n - c_h^{n-1}}^2_{L^2(\Omega)} \leq C \Bigl(1 +  \norm{v_h}_{L^{\infty}(\Omega)} + \norm{v_h}_{L^{\infty}(\Omega)}^2 \Bigr))$. We conclude that  
\begin{eqnarray}
\label{estimvc}
\tau \sum \limits_{n=1}^{T/\tau} \norm{ \divg(S(v_h)^{\frac{1}{2}} {v_c}_h^n)}^2_{L^2(\Omega)} \leq C \Bigl(1 +  \norm{v_h}_{L^{\infty}(\Omega)} + \norm{v_h}_{L^{\infty}(\Omega)}^2 \Bigr).
\end{eqnarray}
\end{proof}

\section{Convergence}
\subsection{Scheme's convergence}
We recall the problem: 
\begin{eqnarray}
\label{problemeP}
\begin{cases}
R \psi \partial_t c + \divg(v c - S(v) \psi \nabla c)= - r(c) + p \chi_{\mathcal{S}}  \text{ in $\Omega_T$}\\
S(v) \nabla c \cdot n = 0 \text{ ou } c=0 \text{ on $\partial \Omega \times (0,T)$} \\
c_{|{t=0}} =c_0 \text{ in $\Omega$.}
\end{cases}
\end{eqnarray}
We chose here a Neumann boundary condition. From now, we assume that one of the following assumptions is satisfied: \\
- $\partial \Omega$ is $\C^2$, $\kappa \in (\C^1(\Omega))^{N \times N}$ ; \\ 
- $\partial \Omega$ is $\C^2$, $\kappa = \kappa^* Id$ with $\kappa^* : \bar{\Omega} \rightarrow \R$ and $\kappa \in \C^1(\Omega)$ ; \\ 
which ensure that solution $v$ of the problem is $L^{\infty}(\Omega_T)$. Moreover, assume that $v_h$, its discretization, is given by the scheme presented in \eqref{discretphi}. We recall that according to Proposition \ref{prop:convBernardi}, we get
$$ \lim \limits_{h \rightarrow 0} \norm{v_h - v}_{(L^2(\Omega_T))^N} = 0 \text{ and } \norm{v_h}_{(L^{\infty}(\Omega_T))^N} \leq C.$$
It follows that \eqref{CFLc} and \eqref{stabilite1}-\eqref{stabilite2} are satisfied. \\
As we will work with the mixed formulation, we use again the flow notation
$$ v_c = -S(v) \psi \nabla c + v c. $$
We introduce the following spaces: 
$$ \mathcal{W} := H^1(0,T;L^2(\Omega)),$$
$$ \mathcal{U} := L^2(0,T;H(\divg;\Omega)).$$
For the fully discretization, we recall the definition of the following discret sub-spaces: $\mathcal{W}_h \subset L^2(\Omega)$ and $\mathcal{U}_h \subset H(\divg;\Omega)$ defined as follows: 
$$ \mathcal{W}_h := \{ c \in L^2(\Omega), c \text{ is constant on each element $\mathcal{O} \in \mathcal{T}_h$} \},$$
$$ \mathcal{U}_h := \{ {v_c} \in H(\divg,\Omega), {v_c}  \text{ is linear on each element $\mathcal{O} \in \mathcal{T}_h$} \}.$$
We define also the following projections:
$$ P_h : L^2(\Omega) \mapsto \mathcal{W}_h, \langle P_h w - w, w_h \rangle = 0 $$
for all $w_h \in \mathcal{W}_h$. Similarly, projection $\Pi_h$ is defined on $(H^1(\Omega))^2$ such that
$$ \Pi_h : (H^1(\Omega))^2 \mapsto \mathcal{U}_h, \langle \divg(\Pi_h {v_c} - v_c), w_h \rangle = 0 $$
for all $w_h \in \mathcal{W}_h$. According to Kumar et al. \cite{Kumar}, this operator can be extended to $H(\divg;\Omega)$ and we get the following estimations: 
\begin{eqnarray}
\label{estimationsprojection}
\begin{cases}
\norm{w-P_h w}_{L^2(\Omega)} \leq C h \norm{w}_{H^1(\Omega)} \text{ for all $w \in H^1(\Omega)$,} \\
\norm{{v_c}-\Pi_h {v_c}}_{L^2(\Omega)} \leq C h \norm{v_c}_{H^1(\Omega)} \text{ for all ${v_c} \in (H^1(\Omega))^N$,} \\
\norm{\divg {v_c}-\divg(\Pi_h {v_c})}_{L^2(\Omega)} \leq C h \norm{v_c}_{H^2(\Omega)} \text{ for all ${v_c} \in (H^2(\Omega))^N$.}  
\end{cases}
\end{eqnarray}
In order to prove the convergence of the discretized problem, we will first prove the convergence for the semi-discretized in time problem, and then we will study the convergence of the fully discretized problem.
We follow the proof of Kumar et al. \cite{Kumar}. Contrary to \cite{Kumar}, we will not consider the ions transport. However, our equations have a second member and the dispersion tensor (depending on the velocity) is taking into account, which complicating the proofs. Results obtained in previous sections on the stability will be useful in these proofs. 
\subsubsection{Mixed in time variational formulation} \label{subsub:semidiscret}
In what follows, we denoted the time step by $\tau$ and $t_n = n\tau$ for $n=1,...,T/\tau$ in order to consider the time discretization (implicit in $c$) with a uniform time step. At each time step $t_n$, we use $c^{n-1} \in L^2(\Omega)$ computed at $t_{n-1}$ in order to find the following approximation $c^n$. The initialization is $c_0$. More specifically, we find $(c^n, {v_c}^n) \in (L^2(\Omega), H(\divg;\Omega))$ satisfying the following (time) semi-discretized problem : \\
\\
\noindent\fbox{\parbox{\linewidth \fboxrule \fboxsep}{Problem $\mathcal{P}^n$ : for $(c^{n-1},p)$ given in $(L^2(\Omega))^2$, finding $(c^n, {v_c}^n) \in (L^2(\Omega), H(\divg;\Omega))$ such that
\begin{eqnarray}
\label{mixte1}
\langle R \psi (c^n - c^{n-1}), w  \rangle 
+ \tau \langle  \divg(S(v)^{\frac{1}{2}}{v_c}^n),  w    \rangle 
+ \tau \langle r(c^{n-1}),w  \rangle 
- \tau \langle p,w  \rangle = 0, \\
\langle S(v)^{-\frac{1}{2}} \psi^{-1} {v_c}^n , u   \rangle - \langle  c^n , \divg(u)   \rangle - \langle S(v)^{-1} \psi^{-1} v c^n , u   \rangle = 0
\label{mixte2}
\end{eqnarray}
for all $(w,u) \in (L^2(\Omega),H(\divg;\Omega))$. }} \\
\\We can prove for $c_n$ the similar estimates obtained for the fully discretized model in the previous section. However, these estimates are not sufficient to pass to the limit.
In order to prove the convergence of the scheme, we need some compactness result. \\
For that purpose, we introduce the following space translation operator: 
$$ \Delta_{\xi} f(\cdot) := f(\cdot) - f(\cdot + \xi) \text{ , $\xi \in \R^N$.} $$
For a given $\xi \in \R^N$, we consider $\Omega_{\xi} \subset \Omega$ such that $\Omega_{\xi} := \{ x \in \Omega, dist(x, \Gamma) > \xi \}$. In this way, translations $\Delta_{\xi} f(x)$ with $x \in \Omega_{\xi}$ are well-defined. We recall that if $u \in H^1(\Omega)$, there exists a constant $C$ such that for all open set $\omega \subset \subset \Omega$ and for all $\xi \in \R^N$ with $\abs{\xi} < dist(\omega, \Omega)$ we get 
\begin{eqnarray}
\label{formulep110} 
\norm{\Delta_{\xi} u - u}_{L^2(\omega)} \leq C \norm{u}_{H^1(\Omega)} \abs{\xi}. 
\end{eqnarray}
We first consider the space translation of $c_h^n$ defined in $\Omega$ and extended by 0 outside $\Omega$. The following lemma allows us to control this translation: 
\begin{lem}
\label{lem:translationc}
We get the following estimate:
\begin{eqnarray}
\label{translationc}
\tau \sum \limits_{n=1}^{T/\tau}\norm{\Delta_{\xi} c^n}^2_{L^2(\Omega_T)} \leq C \abs{\xi}  (1+\norm{v}_{L^{\infty}(\Omega)}+\norm{v}^{2}_{L^{\infty}(\Omega)} ) \leq C \abs{\xi}.
\end{eqnarray}
\end{lem}
\begin{proof}
We do a space translation on equation \eqref{mixte2}. We get:
\begin{eqnarray}
\label{nump124}
\langle \Delta_{\xi}(S(v)^{-\frac{1}{2}} \psi^{-1} {v_c}^n) , u   \rangle - \langle  \Delta_{\xi}(c^n) , \divg(u)   \rangle - \langle \Delta_{\xi}(S(v)^{-1} \psi^{-1} v c^n) , u  \rangle = 0.
\end{eqnarray}
We build an appropriate test function in order to obtain the above estimate. We get $\eta^n$ such that
\begin{eqnarray*}
\begin{cases}
-\Delta \eta^n = \Delta_{\xi} c^n \text{ in $\Omega$,} \\
\eta^n = 0 \text{ on $\Gamma$.}
\end{cases}
\end{eqnarray*}
By taking $u= \nabla \eta^n$ in \eqref{nump124}, we get: 
\begin{eqnarray}
\label{eqdeltatxi}
\langle \Delta_{\xi}(S(v)^{-\frac{1}{2}} \psi^{-1} {v_c}^n) , \nabla \eta^n   \rangle + \langle  \Delta_{\xi}(c^n) , \Delta_{\xi}(c^n)   \rangle - \langle \Delta_{\xi}(S(v)^{-1} \psi^{-1} v c^n) ,\nabla \eta^n   \rangle = 0.
\end{eqnarray}
Noticing that $\eta^n$ satisfies $\norm{\Delta \eta^n}_{L^2(\Omega)} = \norm{\Delta_{\xi}(c^n)}_{L^2(\Omega)}$, and then $\norm{\eta^n}_{H^2(\Omega)} \leq C \norm{\Delta_{\xi}(c^n)}_{L^2(\Omega)}$. \\
This implies that translations of $\nabla \eta^n$ are controlled by those of $c^n$ pursuant to \eqref{formulep110}. More specifically, according to \eqref{formulep110}, we get: 
\begin{align*}
 \norm{\Delta_{\xi_0} (\nabla \eta^n)}_{L^2(\Omega)} = \norm{\nabla(\Delta_{\xi_0} \eta^n)}_{L^2(\Omega)} &\leq C \abs{\xi_0} \norm{\nabla \eta^n}_{H^1(\Omega)}  &&\\
 &\leq C \abs{\xi_0} \norm{\Delta_{\xi_0}(c^n)}_{L^2(\Omega)} \text{ , $\forall \xi_0 \in \R^N$.}&&
 \end{align*}
From \eqref{CFLc}, $\norm{c^n}_{L^2(\Omega)} \leq C$. As $\norm{\Delta_{\xi_0}(c^n)}_{L^2(\Omega)} \leq 2 \norm{c^n}_{L^2(\Omega)}$, we thus get: 
\begin{eqnarray}
\label{3.11}
\norm{\nabla(\Delta_{\xi_0}\eta^n)}_{L^2(\Omega)} \leq C \abs{\xi_0} \text{ , $\forall \xi_0 \in \R^N$.} 
\end{eqnarray}
Furthermore, by adding equation \eqref{eqdeltatxi} for $n=1,...,T/\tau$, we get the following result\footnote{We use also that for all functions $f$ and $g$ defined on $\Omega$ and extended by $0$ outside of $\Omega$, we get $\langle \Delta_{\xi} f,g \rangle = \langle f, \Delta_{-\xi}g \rangle$. Ineed, $\langle \Delta_{\xi} f,g \rangle = \int_{\Omega} f(x) g(x) \dx - \int_{\Omega} f(x+\xi) g(x) \dx = \int_{\R^3} f(x) g(x) \dx - \int_{\R^3} f(x+\xi) g(x) \dx$ (according to the extended by $0$) where we find by variable change that $\int_{\R^3} f(x+\xi)g(x) \dx = \int_{\R^3} f(x)g(x - \xi) d \xi$.}: 
\begin{eqnarray*}
\tau \sum \limits_{n=1}^{T/\tau} \norm{\Delta_{\xi}c^n}^2_{L^2(\Omega)} = \tau \sum \limits_{n=1}^{T/\tau}  \Bigl( S(v)^{-1}\psi^{-1} v c^n , \nabla (\Delta_{-\xi} \eta^n) \Bigr) - \tau  \sum \limits_{n=1}^{T/\tau}  \Bigl( S(v)^{-\frac{1}{2}} \psi^{-1}{v_c}^n , \nabla (\Delta_{-\xi} \eta^n) \Bigr).  
\end{eqnarray*}
According to \eqref{eq5.7.1}, $\abs{S(v)^{-1} \psi^{-1} v c^n} \leq C M_+ \abs{c^n}$. 
Moreover, by using \eqref{eq.etoile}, $\abs{S(v)^{-\frac{1}{2}} \psi^{-1} {v_c}^n} \leq \frac{C v_c^n}{\Bigl(1+\abs{v}\Bigr)^{\frac{1}{2}}}$. Thus
\begin{align*}
 \tau \sum \limits_{n=1}^{T/\tau} \norm{\Delta_{\xi}c^n}^2_{L^2(\Omega)} 
&\leq  C \tau \sum \limits_{n=1}^{T/\tau} \norm{c^n}_{L^2(\Omega)} \norm{\nabla (\Delta_{-\xi}\eta^n)}_{L^2(\Omega)} &&\\
&+ C \tau \sum \limits_{n=1}^{T/\tau}  \norm{ \frac{1}{(1+\abs{v})^{\frac{1}{2}}} \abs{{v_c}^n}}_{L^2(\Omega)} \norm{\nabla (\Delta_{-\xi} \eta^n)}_{L^2(\Omega)}.&&
\end{align*}
According to \eqref{stabilite1}, $\norm{ \frac{1}{(1+\abs{v})^{\frac{1}{2}}} v_c^n}_{L^2(\Omega)} \leq \norm{{v_c}^n}_{L^2(\Omega_T)} \leq C\Bigl(1+\norm{v}_{L^{\infty}(\Omega)}+\norm{v}^{2}_{L^{\infty}(\Omega)}\Bigr)$, and by using \eqref{CFLc} and \eqref{3.11}, we get:  
\begin{eqnarray*}
 \tau \sum \limits_{n=1}^{T/\tau} \norm{\Delta_{\xi}c^n}^2_{L^2(\Omega)} 
\leq C \abs{\xi} \Bigl(1+\norm{v}_{L^{\infty}(\Omega)}+\norm{v}^{2}_{L^{\infty}(\Omega)}\Bigr) \leq C \abs{\xi}
\end{eqnarray*}
which conclude the proof. 
\end{proof} 
We will now prove the time convergence for the semi-discretized scheme. We consider the sequence $\{ (c^n, {v_c}^n), n = 0,...,T/\tau \}$ solution of problem \eqref{mixte1}-\eqref{mixte2}, and we construct a time-continuous approximation by linear interpolation. For $t \in (t_{n-1}, t_n]$, $n=1,...,T/\tau$, we define
\begin{eqnarray}
\label{Ztau}
Z^{\tau}(t) := z^n \frac{(t-t_{n-1})}{\tau} + z^{n-1} \frac{(t_n-t)}{\tau},
\end{eqnarray}
where $Z^{\tau}$ may refer respectively $C^{\tau}$ or ${V_C}^{\tau}$, with respectively $z^n=c^n$ or $z^n=v_c^n$. 
\begin{lem}
\label{lem:estimationsC}
There exists a constant $C >0$ such that for all $\tau$ we get the following estimates: 
\begin{eqnarray}
\label{estimCtau}
 \norm{C^{\tau}}_{L^2(\Omega_T)} + \tau \norm{{V_C}^{\tau}}_{L^2(\Omega_T)} \leq C, \\
\norm{\partial_t C^{\tau}}_{L^2(\Omega_T)} + \norm{\divg(\psi^{-1} S(v)^{\frac{1}{2}} V_C^{\tau})}_{L^2(\Omega_T)} \leq C\Bigl(1+\norm{v}_{L^{\infty}(\Omega)}+\norm{v}^{2}_{L^{\infty}(\Omega)} \Bigr).
\label{estimdtCtau}
\end{eqnarray}
\end{lem}
\begin{proof}
As
\begin{align*}
\norm{C^{\tau}}_{L^2(\Omega_T)}^2 &\leq \norm{\abs{\frac{t-t_{n-1}}{\tau}}\abs{c^{n}} +\abs{\frac{t_{n}-t}{\tau}}\abs{c^{n-1}}}_{L^2(\Omega_T)}^2 \leq  \norm{\abs{c^{n}} +\abs{c^{n-1}}}_{L^2(\Omega_T)}^2 &\\
&\leq 2 \norm{c^n}_{L^2(\Omega_T)}^2 + 2 \norm{c^{n-1}}_{L^2(\Omega_T)}^2,&&
\end{align*} 
according to \eqref{CFLc}, we get
$$ \norm{C^{\tau}}_{L^2(\Omega_T)}^2 \leq 2 \norm{c^n}_{L^2(\Omega_T)}^2 + 2 \norm{c^{n-1}}_{L^2(\Omega_T)}^2 \leq C.$$
Similarly, 
$$ \norm{{V_C}^{\tau}}_{L^2(\Omega_T)}^2 \leq 2 \norm{{v_c}^n}_{L^2(\Omega_T)}^2 + 2 \norm{{v_c}^{n-1}}_{L^2(\Omega_T)}^2 \leq \frac{C}{\tau}. $$
Thus \eqref{estimCtau}.  \\
In order to estimate $\norm{\partial_t C^{\tau}}_{L^2(\Omega_T)}$, we denote for each $\tau \in (t_{n-1}, t_n]$, $\partial_t C^{\tau} = \frac{c^n - c^{n-1}}{\tau}$, which implies that 
$$ \int_{0}^{T} \norm{\partial_t C^{\tau}}_{L^2(\Omega_T)}^2 \dt = \sum \limits_{n=1}^{T/\tau} \int_{t_{n-1}}^{t_n} \frac{1}{\tau^2} \norm{c^n -c^{n-1}}_{L^2(\Omega_T)}^2 \dt \leq \sum \limits_{n=1}^{T/\tau} \frac{1}{\tau} \norm{c^n - c^{n-1}}_{L^2(\Omega_T)}^2.$$
From \eqref{stabilite1}, we thus get
\begin{multline*} \int_{0}^{T} \norm{\partial_t C^{\tau}}^2 \dt \leq C \frac{T}{\tau} \Bigl(1+\norm{v}_{L^{\infty}(\Omega)}+\norm{v}^{2}_{L^{\infty}(\Omega)} \Bigr)
 \leq C \Bigl(1+\norm{v}_{L^{\infty}(\Omega)}+\norm{v}^{2}_{L^{\infty}(\Omega)} \Bigr).
\end{multline*}
Similarly, 
\begin{multline*}
 \norm{{V_C}^{\tau}}^2_{L^2(\Omega)} = \int_{0}^{T} \norm{\partial_t {V_C}^{\tau}}_{L^2(\Omega)}^2 \dt = \sum \limits_{n=1}^{T/\tau} \int_{t_{n-1}}^{t_n} \frac{1}{\tau^2} \norm{{v_c}^n -{v_c}^{n-1}}_{L^2(\Omega)}^2 \dt \\
 \leq \sum \limits_{n=1}^{T/\tau} \frac{1}{\tau} \norm{{v_c}^n - {v_c}^{n-1}}_{L^2(\Omega)}^2 \leq  C\Bigl(1+\norm{v}_{L^{\infty}(\Omega)}+\norm{v}^{2}_{L^{\infty}(\Omega)} \Bigr).
 \end{multline*}
Finally, notice that
$$ \divg(S(v)^{\frac{1}{2}} {V_C}^{\tau}) =  \divg \Bigl(S(v)^{\frac{1}{2}} {V_C}^{n-1}\Bigr) + \frac{t- t_{n-1}}{\tau}  \divg \Bigl(S(v)^{\frac{1}{2}} ({V_C}^{n}-{V_C}^{n-1})\Bigr).$$
According to \eqref{stabilite2} we prove the estimate of $\divg(S(v)^{\frac{1}{2}} {V_C}^{\tau})$ in $L^2(\Omega_T)$.
\end{proof} 
Estimates of previous lemma ensure the existence of functions $c^{\#}, v_c^{\#}$, and the existence of a sub-sequence $\tau \rightarrow 0$ such that
\begin{itemize}
\item $C^{\tau}  \rightharpoonup c^{\#}$ weakly in $L^2(\Omega_T) \cap H^1(0,T;L^2(\Omega))$, \\
\item ${V_C}^{\tau} \rightharpoonup v_c^{\#}$,  $\divg(S(v)^{\frac{1}{2}} V_C^{\tau}) \rightharpoonup \divg(S(v)^{\frac{1}{2}} v_c^{\#})$ weakly in $L^2(0,T;L^2(\Omega)^N)$.
\end{itemize}
In order to obtain a strong convergence, we use the translated estimates as in Lemma \ref{lem:translationc}. We get the following result: 
\begin{lem}
\label{lem:convergenceforteC}
Sequence $C^{\tau}$ strongly converges to $c^{\#}$ in $L^2(\Omega_T)$ when $\tau \rightarrow 0$. 
\end{lem}
\begin{proof}
We will use the Riesz-Frechet-Kolmogorov theorem. 
As $\partial_t C^{\tau} \in L^2(\Omega_T)$, the time translation is controlled. It remains to control the space translation: 
$$ \mathcal{I}_{\xi} := \int_{0}^{T} \int_{\Omega} \abs{\Delta_{\xi} C^{\tau}}^2 \dx \dt \rightarrow 0 \text{ when $\abs{\xi} \rightarrow 0$.}$$
According to the definition of $C^{\tau}$, we get 
\begin{multline*}
\mathcal{I}_{\xi} = \int_{0}^{T} \int_{\Omega} \abs{C^{\tau}(x) - C^{\tau}(x+\xi)}^2 \dx \dt = \\
\sum \limits_{n=1}^{T/\tau} \int_{t_{n-1}}^{t_n} \int_{\Omega_{\xi}} | c^{n}(x) \frac{t-t_{n-1}}{\tau} + c^{n-1}(x) \frac{t_n-t}{\tau} - c^{n}(x+\xi) \frac{t-t_{n-1}}{\tau}- c^{n-1}(x+\xi) \frac{t_n-t}{\tau}| ^2 \dx \dt  \\
=  \sum \limits_{n=1}^{T/\tau} \int_{t_{n-1}}^{t_n} \int_{\Omega_{\xi}} \abs{\Delta_{\xi} c^n \frac{t-t_{n-1}}{\tau} + \Delta_{\xi} c^{n-1} \frac{t_n -t}{\tau}}^2 \dx \dt.
\end{multline*} 
Thus
$$ \abs{\mathcal{I}_{\xi}} \leq \sum \limits_{n=1}^{T/\tau}  \tau \Bigl(2 \norm{\Delta_{\xi} c^n}^2_{L^2(\Omega_T)} + 2 \norm{\Delta_{\xi} c^{n-1}}^2_{L^2(\Omega_T)}\Bigr).$$
By using Lemma \ref{lem:translationc}, we get that
$$\abs{\mathcal{I}_{\xi}} \leq C \abs{\xi}\Bigl(1+\norm{v}_{L^{\infty}(\Omega)}+\norm{v}^{2}_{L^{\infty}(\Omega)} \Bigr),$$ 
where $C$ in independently of $\tau$ and $h$. \\
We thus control the space translation of $c^n$ and according to Riesz-Frechet-Kolmogorov theorem, $\{ C^{\tau}, \tau >0 \}$ is compact. We thus conclude the strong convergence in $L^2(\Omega_T)$ of $C^{\tau}$ to $c^{\#}$. 
\end{proof} 
Now we have prove the strong convergence of $C^{\tau}$, we can pass to the limit. 
\begin{thm}
\label{thm:solutionfaible}
The sequence $(C^{\tau},V_C^{\tau})$ converges to the solution $(c,v_c)$ of \eqref{problemeP} when $\tau \rightarrow 0$. 
\end{thm}
\begin{proof}
Function $(C^{\tau},V_C^{\tau})$ satisfies
\begin{multline}
\label{3.20}
 \langle R \psi  \partial_t C^{\tau}, w \rangle +  \langle \divg(S(v)^{\frac{1}{2}} {V_C}^{\tau}), w \rangle + \langle  r(C^{\tau}), w \rangle -  \langle p, w \rangle= \\ 
 \langle \divg(S(v)^{\frac{1}{2}} ({V_C}^{\tau} -  {v_c}^n)), w  \rangle  + \langle  r(C^{\tau}) - r(c^{n-1}), w \rangle 
 \end{multline}
 and
 \begin{multline}
  \label{3.22}
 \langle  S(V)^{-\frac{1}{2}}\psi^{-1} {V_C}^{\tau}, u \rangle - \langle C^{\tau} ,\divg(u) \rangle - \langle S(V)^{-1} \psi^{-1} V C^{\tau} , u \rangle  = \\
  \langle  S(V)^{-\frac{1}{2}}\psi^{-1} ({V_C}^{\tau} - {v_c}^n), u \rangle - \langle C^{\tau} -c^n ,\divg(u) \rangle - \langle S(V)^{-1} \psi^{-1} (V C^{\tau} - V c^n) , u \rangle
\end{multline}
for all $(w,u) \in (L^2(0,T;H^1_0(\Omega)), \mathcal{S})$. 
We first consider \eqref{3.20} and thanks to Lemma  \ref{lem:convergenceforteC}, the left member converges to $\langle R \psi \partial_t c^{\#},w \rangle + \langle \divg(S(v)^{\frac{1}{2}} v_c^{\#}),w\rangle + \langle r(c^{\#}),w\rangle - \langle p,w \rangle$. Lets prove that the right member removes when $\tau \rightarrow 0$. We denoted the two terms of the right member respectively $I_1$ and $I_2$. Integrating $I_1$ by parts, which is allowed according to the choice of $w \in L^2(0,T;H^1_0(\Omega))$, we get 
\begin{align*} 
I_1 &= \sum \limits_{n=1}^{T/\tau}  \int_{t_{n-1}}^{t_n} S(v)^{\frac{1}{2}}({V_C}^{\tau} - {v_c}^{n}) \cdot \nabla w \dx \dt &&\\
&= \sum \limits_{n=1}^{T/\tau}  \int_{t_{n-1}}^{t_n} S(v)^{\frac{1}{2}}({v_c}^n \frac{t-t_{n-1}}{\tau} +{v_c}^{n-1} \frac{t_n-t}{\tau} - {v_c}^{n}) \cdot \nabla w \dx \dt &&\\
&= \sum \limits_{n=1}^{T/\tau}  \int_{t_{n-1}}^{t_n} \frac{t-t_n}{\tau} S(v)^{\frac{1}{2}}({v_c}^n - {v_c}^{n-1}) \cdot \nabla w \dx \dt.&&
\end{align*}
When $t \in [t_{n-1}, t_n[$, $ \frac{t-t_n}{\tau} < 1$. Thus, according to Cauchy-Schwarz inequality, we get
\begin{align*}
\abs{I_1} &\leq \tau^{\frac{1}{2}} \Bigl( \sum \limits_{n=1}^{T/\tau} \norm{S(v)^{\frac{1}{2}}({v_c}^n - {v_c}^{n-1})}_{L^2(\Omega)}^2  \Bigr)^{\frac{1}{2}}  \Bigl( \int_{0}^{T} \norm{\nabla w}_{L^2(\Omega)}^2  \Bigr)^{\frac{1}{2}} &&\\ 
&\leq  \tau^{\frac{1}{2}} C  \Bigl(1+\norm{v}_{L^{\infty}(\Omega)}+\norm{v}^{2}_{L^{\infty}(\Omega)} \Bigr) \norm{\nabla w}_{L^2(\Omega)} && \\
& \leq C \tau^{\frac{1}{2}}   \rightarrow 0 \text{ when $\tau \rightarrow 0$} &&
\end{align*}
thanks to l'estimate \eqref{stabilite1}. \\
We get
\begin{align*}
\abs{I_2} &= \abs{\langle r(C^{\tau})-r(c^{n-1}),w\rangle} &&\\
&\leq \Bigl(C \sum \limits_{n=1}^{T/\tau} \tau \norm{r(C^{\tau})-r(c^{n-1})}_{L^2(\Omega)}^2  \Bigr)^{\frac{1}{2}} \Bigl(\sum \limits_{n=1}^{T/\tau} \int_{t_{n-1}}^{t_n} \norm{w}_{L^2(\Omega)}^2 \dt \Bigr)^{\frac{1}{2}}.&&
\end{align*}
As $r$ is assumed derivable with a bounded derivative, we get 
\begin{align*}
\abs{I_2} &\leq \tau^{\frac{1}{2}} \norm{r'}_{L^{\infty}(\R)} \Bigl( \sum_{n=1}^{T/\tau} \norm{C^{\tau}-c^{n-1}}_{L^2(\Omega)}^2 \Bigr)^{\frac{1}{2}}  \Bigl( \sum_{n=1}^{T/\tau} \int_{t_{n-1}}^{t_n} \norm{w}_{L^2(\Omega)}^2 \dt \Bigr)^{\frac{1}{2}} &&\\
&\leq C \tau^{\frac{1}{2}}  \norm{r'}_{L^{\infty}(\R)}  \Bigl( \sum_{n=1}^{T/\tau} \norm{c^{n}-c^{n-1}}_{L^2(\Omega)}^2 \Bigr)^{\frac{1}{2}}   \norm{w}_{L^2(\Omega)}. &&
\end{align*}
From \eqref{stabilite1}, we get
\begin{align*}
\abs{I_2} &\leq C \tau^{\frac{1}{2}} \norm{r'}_{L^{\infty}(\R)} \Bigl( \tau \Bigl(1+\norm{v}_{L^{\infty}(\Omega)}+\norm{v}^{2}_{L^{\infty}(\Omega)} \Bigr) \Bigr)^{\frac{1}{2}}   \norm{w}_{L^2(\Omega)}&& \\
&\leq C \tau \norm{r'}_{L^{\infty}(\R)} \Bigl(1+\norm{v}_{L^{\infty}(\Omega)}+\norm{v}^{2}_{L^{\infty}(\Omega)} \Bigr)^{\frac{1}{2}}   \norm{w}_{L^2(\Omega)}&& \\
&\leq C \tau \rightarrow 0 \text{ when $\tau \rightarrow 0$.} &&
\end{align*}
Then, $\abs{I_2} \rightarrow 0$. 
Thus
\begin{eqnarray*}
\label{3.20bis}
\lim \limits_{\tau \rightarrow 0} \Bigl(\langle R \psi  \partial_t C^{\tau}, w \rangle +  \langle \divg(S(v)^{\frac{1}{2}} {V_C}^{\tau}), w \rangle + \langle  r(C^{\tau}), w \rangle -  \langle p, w \rangle \Bigr)=0,
 \end{eqnarray*}
 meaning 
 \begin{eqnarray}
\label{3.20ter}
 \langle R \psi  \partial_t c^{\#}, w \rangle +  \langle \divg(S(v)^{\frac{1}{2}} v_c^{\#}), w \rangle + \langle  r(c^{\#}), w \rangle -  \langle p, w \rangle=0.
 \end{eqnarray}
We now focus on \eqref{3.22}. We denoted the three first terms of the right member respectively $I_4$, $I_5$ and $I_6$.
We get
\begin{align*}
\abs{I_4} & \leq  \sum \limits_{n=1}^{T/\tau}  \int_{t_{n-1}}^{t_n} S(v)^{-\frac{1}{2}} \psi^{-1} ({V_C}^{\tau} - {v_c}^n) \cdot u \dx \dt &&\\
&\leq \sum \limits_{n=1}^{T/\tau}  \int_{t_{n-1}}^{t_n} \abs{\frac{t-t_n}{\tau} S(v)^{-\frac{1}{2}} \psi^{-1} ({v_c}^{n} - {v_c}^{n-1}) \cdot u }\dx \dt &&\\
& \leq \tau^{\frac{1}{2}} \Bigl(\sum \limits_{n=1}^{T/\tau} \norm{S(v)^{-\frac{1}{2}} \psi^{-1}({v_c}^n - {v_c}^{n-1})}_{L^2(\Omega)}^2 \Bigr)^{\frac{1}{2}} \Bigl(\int_{0}^{T} \norm{u}_{L^2(\Omega)}^2 \dt \Bigr)^{\frac{1}{2}} \rightarrow 0 \text{ when $\tau \rightarrow 0$}&&
\end{align*}
according to \eqref{stabilite1}. \\
Similarly, we find $\abs{I_5}$ and $\abs{I_6} \rightarrow 0$. \\
Thus
\begin{eqnarray*}
\label{3.24}
\lim \limits_{\tau \rightarrow 0} \Bigl( \langle  S(v)^{-\frac{1}{2}}\psi^{-1} {V_C}^{\tau}, u \rangle - \langle C^{\tau} ,\divg(u) \rangle - \langle S(v)^{-1} \psi^{-1} V C^{\tau} , u \rangle \Bigr) =0
\end{eqnarray*} 
meaning
\begin{eqnarray}
\label{3.24bis}
\langle  S(v)^{-\frac{1}{2}}\psi^{-1} v_c^{\#}, u \rangle - \langle c^{\#} ,\divg(u) \rangle - \langle S(v)^{-1} \psi^{-1} v c^{\#} , u \rangle  = 0.
\end{eqnarray} 
According to \eqref{3.20ter} and \eqref{3.24bis}, $(c^{\#},v_c^{\#})$ is a weak solution of \eqref{problemeP}. The solution of the problem being unique, we get $(c^{\#},v_c^{\#})=(c,v_c)$ and the sequence $(C^{\tau},V_C^{\tau})$ converges to this limit. 

\end{proof}
We have proved the convergence of the solution for the time discretized model. We are now going to study the fully discretized problem, in time and in space, and prove the convergence of the solution. 

\subsubsection{Fully discretization} \label{subsub:discretisationtoatle}
The problem is initialized with $c_h^0 = c_0$, with $n=1,...,N$. We find the approximation $(c_h^n,{v_c}_h^n)$ of $(c(t_n), {v_c}(t_n))$ at $t=t_n$ solution of the following problem: \\ 
Problem $\mathcal{P}_h^n$: for a given $c_h^{n-1} \in \mathcal{W}_h$, finding $(c_h^n,{v_c}_h^n) \in (\mathcal{W}_h, \mathcal{U}_h)$ satisfiying
\begin{eqnarray}
\label{mixte1bis}
\langle R \psi (c_h^n - c_h^{n-1}), w \rangle 
+ \tau \langle  \divg(S(v_h)^{\frac{1}{2}}{v_c}_h^n),  w    \rangle 
+ \tau \langle r(c_h^{n-1}),w  \rangle 
- \tau \langle p ,w  \rangle = 0, \\
\langle S(v_h)^{-\frac{1}{2}} \psi^{-1} {v_c}_h^n , u   \rangle - \langle  c_h^n , \divg(u)   \rangle - \langle S(v_h)^{-1} \psi^{-1} v_h c_h^n , u   \rangle = 0
\label{mixte2bis}
\end{eqnarray}
for all $(w,u) \in \mathcal{W}_h \times \mathcal{U}_h$. \\
The approach in order to prove the convergence of the solution is the same as the one used for the semi-discretized model. We introduce $Z_h^{\tau}$ the approximation of $Z_h^{\tau}$ by interpolation defined by
$$ Z_h^{\tau} = z_h^{n} \frac{t-t_{n-1}}{\tau} + z_h^{n-1} \frac{t_n -t}{\tau}$$ 
where $Z_h^{\tau}$ may refer to $C_h^{\tau}$ or ${V_C}_h^{\tau}$. \\
Ensures the following lemma: 
\begin{lem}
\label{lem:majorationCntau}
There exists a constant $C>0$ such that for all $\tau$ and $h$ we get the following estimation 
\begin{eqnarray}
\label{majorationCntau}
\norm{C_h^{\tau}}^2_{L^2(\Omega_T)} + \tau \norm{{V_C}_h^{\tau}}^2_{L^2(\Omega_T)}  \leq C, \\
\norm{\partial_t C_h^{\tau}}^2_{L^2(\Omega_T)}  + \norm{\divg(\psi^{-1}S(v_h)^{\frac{1}{2}} {V_C}_h^{\tau})}^2_{L^2(\Omega_T)}   \leq C\Bigl(1+\norm{v}_{L^{\infty}(\Omega)}+\norm{v}^{2}_{L^{\infty}(\Omega)} \Bigr) \leq C.
\label{majorationCntau2}
\end{eqnarray}
\end{lem}
\begin{proof}
The proof is similar to those of Lemma \ref{lem:estimationsC}.
\end{proof}
\begin{lem}
\label{lem:convergencefaibleChtau}
There exists a sub-sequence $\tau \rightarrow 0$ and functions $c^{\#}$ and $v_c^{\#}$ such that: 
\begin{align*}
C_h^{\tau}  & \rightharpoonup c^{\#} \text{ weakly in $L^2(\Omega_T)$,} &&\\
\partial_t C_h^{\tau} & \rightharpoonup \partial_t c^{\#} \text{ weakly in $L^2(\Omega_T)$,} &&\\
{V_C}_h^{\tau}  & \rightharpoonup v_c^{\#} \text{ weakly in $L^2(0,T;L^2(\Omega)^N)$,} &&\\
\divg(S(v_h)^{\frac{1}{2}} {V_C}_h^{\tau}) & \rightharpoonup \divg(S(v)^{\frac{1}{2}} v_c^{\#}) \text{ weakly in $L^2(0,T;L^2(\Omega)^N)$.}&&
\end{align*}
\end{lem}
\begin{proof}
The proof is similar to those in the semi-discrete case. For the last convergence, we use that $v_h \rightarrow v$ strongly in $L^2(\Omega_T)$. 
\end{proof} 
We denoted E the set of triangles bounds $\mathcal{T}_h$. Moreover we get $E = E_{int} \cup E_{ext}$ with $E_{int} = E \backslash E_{ext}$. We use the following notation: 
$$ \abs{T} \text{ area of $T \in \mathcal{T}_h$, $x_i$ the center of the circumscribed cercle to $T$,} $$
$$ \abs{\mathcal{T}_h}=\max \limits_{T_i \in \mathcal{T}_h} \abs{T_i}, $$
$$ \ell_{ij} \text{ boundary between $T_i$ and $T_j$, $d_{ij}$ the distance from $x_i$ to $\ell_{ij}$,  $\sigma_{ij}=\frac{\abs{\ell_{ij}}}{d_{ij}}$.} $$
We defined the following inner products for all $c_h^n$, $w_h^n \in \mathcal{U}_h$ : 
\begin{eqnarray}
\label{innerproduct}
(c_h^n,w_h^n)_h := \sum \limits_{T_i \in \mathcal{T}_h} \abs{T_i} c_{h,i}^n w_{h,i}^n \text{, } \hspace{3mm}
(c_h^n,w_h^n)_{1,h} := \sum \limits_{l_{ij} \in E} \abs{\sigma_{ij}} (c_{h,i}^n -c_{h,j}^h)(w_{h,i}^n - w_{h,j}^n).
\end{eqnarray}
The discrete inner product gives the discrete norm $H_0^1$: 
\begin{eqnarray}
\label{normeH10discrete}
\norm{c_h^n}^2_{1,h}= \sum \limits_{l_{ij} \in E} \abs{\sigma_{ij}} (c_{h,i}^n - c_{h,j}^n)^2.
\end{eqnarray}

\begin{lem}
\label{lem:classique}
Let $\Omega$ an open set of $\R^N$, $N=2$ or $3$, and $\mathcal{T}_h$ a mesh. For a $c$ defined in $\Omega$ ans extended in $\bar{c}$ by 0 outside of $\Omega$, we get
\begin{eqnarray}
\label{majorationclassique}
\norm{\Delta_{\xi}\bar{c}}^2_{L^2(\R^N)} \leq   \norm{c}^2_{1,h} \abs{\xi} (\abs{\xi} + C \abs{\mathcal{T}_h})  \text{ for all $\xi \in \R^N$.}
\end{eqnarray}
\end{lem}
\begin{proof} 
We refer to the proof of lemma 4 in Eymard et al. \cite{Gallouet}.
\end{proof}

\begin{lem}
\label{lem:majornormH10}
For a sequence $c_h^n$, we get the following inequality: 
\begin{eqnarray}
\label{majornormH10}
\norm{c_h^n}_{1,h} \leq C \Bigr( \norm{{v_c}_h^n}_{L^2(\Omega)} + \norm{c_h^n}_{L^2(\Omega)} \Bigl)
\end{eqnarray}
with $C$ independently of $h$ and $n$.
\end{lem}
\begin{proof}
We defined$f_h^n$ by given its value on every element $T_i \in \mathcal{T}_h$ : 
\begin{eqnarray}
\label{Tifnh} \abs{T_i} {f_h^n}_{|{T_i}} := \sum \limits_{\ell_{ij}} \frac{\abs{\ell_{ij}}}{d_{ij}} (c_{h,i}^n - c_{h,j}^n).
\end{eqnarray}
According to the definition of $\norm{.}^2_{1,h}$, and because $c_h^n$ is constant by mesh: 
\begin{multline}
\label{fnhchn} (f_h^n,c_h^n) = \sum \limits_{T_i \in \mathcal{T}_h} \abs{T_i} {f_h^n}_{|{T_i}} {c_h^n}_{|{T_i}}=\sum \limits_{T_i \in \mathcal{T}_h} \Bigl( \sum \limits_{\ell_{ij}} \frac{\abs{\ell_{ij}}}{d_{ij}} (c_{h,i}^n - c_{h,j}^n) \Bigr) c_{h,i}^n \\
= \sum \limits_{\ell_{ij} \in E} \frac{\abs{\ell_{ij}}}{d_{ij}} \abs{c_{h,i}^n - c_{h,j}^n}^2 = \norm{c_h^n}^2_{1,h}.
\end{multline}
We see that: 
\begin{align*}
\sum \limits_{T_i \in \mathcal{T}_h} \Bigl( \sum \limits_{\ell_{ij}} \frac{\abs{\ell_{ij}}}{d_{ij}} (c_{h,i}^n - c_{h,j}^n) \Bigr) c_{h,i}^n &= 
\frac{\abs{\ell_{12}}}{d_{12}}(c_{h,1}^n - c_{h,2}^n) c_{h,1}^n + \frac{\abs{\ell_{13}}}{d_{13}}(c_{h,1}^n - c_{h,3}^n) c_{h,1}^n &&\\
&+ \frac{\abs{\ell_{12}}}{d_{12}}(c_{h,2}^n - c_{h,1}^n) c_{h,2}^n + \frac{\abs{\ell_{24}}}{d_{24}}(c_{h,2}^n - c_{h,4}^n) c_{h,2}^n &&\\
&+ \frac{\abs{\ell_{13}}}{d_{13}}(c_{h,3}^n - c_{h,1}^n) c_{h,3}^n + \frac{\abs{\ell_{34}}}{d_{34}}(c_{h,3}^n - c_{h,4}^n) c_{h,3}^n &&\\
&+ \frac{\abs{\ell_{24}}}{d_{24}}(c_{h,4}^n - c_{h,2}^n) c_{h,4}^n + \frac{\abs{\ell_{34}}}{d_{34}}(c_{h,4}^n - c_{h,3}^n) c_{h,4}^n &&\\
&= \sum \limits_{\ell_{ij} \in E} \frac{\abs{\ell_{ij}}}{d_{ij}} (c_{h,i}^n - c_{h,j}^n)^2.&&
\end{align*}
Moreover, thanks to Cauchy-Schwarz inequality: 
$$ \norm{f_h^n}^2_{L^2(\Omega)} = \sum \limits_{i} \abs{T_i \in \mathcal{T}_h} \abs{{f_h^n}_{|{T_i}} }^2 \leq \Bigl(\sum \limits_{\ell_{ij} \in E} \frac{\abs{\ell_{ij}}}{d_{ij}} (c_{h,i}^n - c_{h,j}^n)^2 \frac{1}{\abs{T_i}} \Bigr) \Bigl( \sum \limits_{\ell_{ij}} \frac{\abs{\ell_{ij}}}{d_{ij}}\Bigr)$$
which implies that $\abs{f_h^n} \leq \abs{c_h^n}_{1,h}$. Furthermore, as $f_h^n \in L^2(\Omega)$, there exists $\beta_h \in U_h$ such that
\begin{eqnarray}
\label{4.14}
\divg{\beta_h} = f_h^n \text{ in $\Omega$} \\
\beta_h = 0 \text{ on $\Gamma$.}
\label{4.15}
\end{eqnarray}
According to the bounds of $f_h^n$, we also have that $\norm{\beta_h}_{L^2(\Omega)} \leq C \norm{f_h^n}_{L^2(\Omega)} \leq C \norm{c_h^n}_{1,h}$. \\
With \eqref{4.14} in \eqref{fnhchn}, we find that $(c_h^n, \divg(\beta_h))= (c_h^n, f_h^n) = \norm{c_h^n}^2_{1,h}$. We chose the test function $u=\beta_h$ in \eqref{mixte2bis} and we get
\begin{multline*}
 \norm{c_h^n}^2_{1,h} = (c_h^n,\divg(\beta_h )) = (S(v_h)^{-\frac{1}{2}}\psi^{-1} {v_c}_h^n, \beta_h) - (S(v_h)^{-1}\psi^{-1} v_h c_h^n, \beta_h) \\
 \leq \psi_{-}^{-1} S_m^{-\frac{1}{2}} \norm{{v_c}_h^n}_{L^2(\Omega)} \norm{\beta_h}_{L^2(\Omega)} + \psi_{-}^{-1} S_m^{-1} \norm{v_h}_{L^{\infty}(\Omega)} \norm{c_h^n}_{L^2(\Omega)} \norm{\beta_h}_{L^2(\Omega)} \\
 \leq  C \norm{{v_c}_h^n}_{L^2(\Omega)} \norm{c_h^n}_{1,h} + C \norm{c_h^n}_{L^2(\Omega)} \norm{c_h^n}_{1,h}.
 \end{multline*}
Thus
 $$ \norm{c_h^n}_{1,h}  \leq C \Bigl(   \norm{{v_c}_h^n}_{L^2(\Omega)} +  \norm{c_h^n}_{L^2(\Omega)}   \Bigr).  $$
\end{proof}

\begin{lem}
\label{lem:Ccvfort}
There exists a subsequence of $C^{\tau}_h$ which strongly converge in $L^2(0,T;L^2(\Omega))$ when $(\tau,h) \rightarrow (0,0)$. 
\end{lem}
\begin{proof}
As $\partial_t C^{\tau}_h \in L^2$, the time translation is controlled. It remains to consider the space translation. Taking \eqref{majornormH10} and adding for $n=1,...,T/\tau$, we get: 
\begin{multline}
\label{4.16}
\tau \sum \limits_{n=1}^{T/\tau} \norm{c_h^n}^2_{1,h} \leq C \tau \sum \limits_{n=1}^{T/\tau} \Bigl( \norm{{v_c}_h^n}^2 + \norm{c_h^n}^2 \Bigr) 
\leq C \Bigl(1+\norm{v}_{L^{\infty}(\Omega)}+\norm{v}^{2}_{L^{\infty}(\Omega)} \Bigr) \leq C
\end{multline}
according to \eqref{CFLc} and \eqref{stabilite1}. We extend $c_h^n$ by 0 outside of $\Omega$ (we keep the same notation for more simplicity) and we use Lemma \ref{lem:classique} in order to control the translation by the norm $\norm{.}_{1,h}$ : 
$$ \tau \sum \limits_{n=1}^{T/\tau} \norm{\Delta_{\xi} c_h^n }^2_{L^2(\R^2)} =  \tau \sum \limits_{n=1}^{T/\tau} \norm{c_h^n(\cdot + \xi) - c_h^n}^2_{L^2(\R^2)} \leq \tau C \sum \limits_{n=1}^{T/\tau} \norm{c_h^n}^2_{1,h} \abs{\xi} (\abs{\xi} + \abs{\mathcal{T}_h}).$$
According to \eqref{4.16}
$$ \tau \sum \limits_{n=1}^{T/\tau} \norm{\Delta_{\xi} c_h^n }^2_{L^2(\R^2)}  \leq C \abs{\xi} (\abs{\xi} + \abs{\mathcal{T}_h}) \Bigl(1+\norm{v}_{L^{\infty}(\Omega)}+\norm{v}^{2}_{L^{\infty}(\Omega)} \Bigr) \leq C \abs{\xi} (\abs{\xi} + \abs{\mathcal{T}_h})$$
which gives the following estimate for $C_h^{\tau}$ : 
$$ \tau \sum \limits_{n=1}^{T/\tau} \norm{C_h^{\tau}(\cdot + \xi) - C_h^{\tau}}^2_{L^2(\R^2)} \leq C \abs{\xi} (\abs{\xi} + \abs{\mathcal{T}_h}). $$
According to the Riesz-Frechet-Kolmogorov compactness theorem,  $\{C_h^{\tau}, h>0, \tau >0 \}$ is compact, which drives the conclusion. 
\end{proof} 
The strong convergence of $C_h^{\tau}$ being obtained, we can now pass to the limit. 
\begin{lem}
\label{lem:solfaibleconv}
The limit $(c^{\#},v_c^{\#})$ is a weak solution of \eqref{problemeP}.
\end{lem}
\begin{proof}
Function $(C^{\tau},V_C^{\tau})$ satisfies \eqref{3.20} which is recall below: 
\begin{multline}
\label{5.76bis}
 \langle R \psi  \partial_t C^{\tau}_h, w \rangle +  \langle \divg(S(v_h)^{\frac{1}{2}} {V_C}^{\tau}_h), w \rangle + \langle  r(C^{\tau}_h), w \rangle -  \langle p , w \rangle= \\ 
 \langle \divg(S(v_h)^{\frac{1}{2}} ({V_C}^{\tau}_h -  {v_c}_h^n)), w  \rangle  + \langle  r(C^{\tau}_h) - r(c^{n-1}_h), w \rangle 
 \end{multline}
for all $w \in L^2(0,T;H^1_0(\Omega))$. It thus also satisfies 
\begin{multline*}
 \langle R \psi  \partial_t C^{\tau}_h, w-w_h \rangle +  \langle \divg(S(v_h)^{\frac{1}{2}} {V_C}^{\tau}_h), w-w_h \rangle + \langle  r(C^{\tau}_h), w -w_h \rangle -  \langle p , w-w_h  \rangle= \\ 
 \langle \divg(S(v_h)^{\frac{1}{2}} ({V_C}^{\tau}_h -  {v_c}_h^n)), w-w_h  \rangle  + \langle  r(C^{\tau}_h) - r(c^{n}_h), w-w_h \rangle 
 \end{multline*}
meaning 
\begin{multline}
\label{5.76ter}
 \langle R \psi  \partial_t C_h^{\tau}, w-w_h \rangle +  \langle \divg(S(v_h)^{\frac{1}{2}} {V_C}^{\tau}_h), w-w_h \rangle + \langle  r(C^{\tau}_h), w -w_h \rangle -  \langle p , w-w_h  \rangle \\ 
+ \langle \divg(S(v_h)^{\frac{1}{2}} ({V_C}^{\tau}_h -  {v_c}^n_h)), w_h-w  \rangle  + \langle  r(C^{\tau}_h) - r(c^{n}_h), w_h-w \rangle = 0
 \end{multline} 
for all $w \in L^2(0,T;H^1_0(\Omega))$ and where $w_h$ is the projection of $w$ on $\mathcal{W}_h$, $w_h=P_h w$, previously introduced. By adding \eqref{5.76ter} to the right member of \eqref{5.76bis}, we get 
\begin{align*}
&\int_{0}^{T} (R \psi \partial_t C_h^{\tau},w) \dt  + \int_{0}^{T} (\divg(S(v_h)^{\frac{1}{2}} {V_C}_h^{\tau}),w) \dt + \int_{0}^{T} (r(C_h^{\tau}),w) \dt - \int_{0}^{T} (p ,w) \dt  =&&\\
&\sum \limits_{n=1}^{T/\tau} \int_{t_{n-1}}^{t_n} (R \psi \partial_t C_h^{\tau}, w - w_h) \dt + \sum \limits_{n=1}^{T/\tau} \int_{t_{n-1}}^{t_n} (\divg(S(v_h)^{\frac{1}{2}} {V_C}_h^{\tau}) - \divg(S(v_h)^{\frac{1}{2}} {v_c}_h^{n}) , w ) \dt && \\
& +  \sum \limits_{n=1}^{T/\tau} \int_{t_{n-1}}^{t_n} (\divg(S(v_h)^{\frac{1}{2}} {V_C}_h^{\tau}) , w -w_h ) \dt && \\
&+ \sum \limits_{n=1}^{T/\tau} \int_{t_{n-1}}^{t_n} (\divg(S(v_h)^{\frac{1}{2}} {V_C}_h^{\tau}) - \divg(S(v_h)^{\frac{1}{2}} {v_c}_h^{n}) , w_h -w) \dt  &&\\
&+ \sum \limits_{n=1}^{T/\tau} \int_{t_{n-1}}^{t_n} (r(C_h^{\tau})-r(c_h^n) , w ) \dt + \sum \limits_{n=1}^{T/\tau} \int_{t_{n-1}}^{t_n} (r(C_h^{\tau}) , w-w_h ) \dt && \\
& + \sum \limits_{n=1}^{T/\tau} \int_{t_{n-1}}^{t_n} (r(C_h^{\tau})-r(c_h^n) , w_h -w ) \dt -  \sum \limits_{n=1}^{T/\tau} \int_{t_{n-1}}^{t_n} (p , w-w_h ) \dt.&&
\end{align*}
Noticing that we assume a $H^1$ regularity in space for the test functions $w$. We will use it in order to control terms $\norm{w-w_h}_{L^2(\Omega_T)}$ thanks to \eqref{estimationsprojection}. \\
We denoted terms of the right members respectively $I_i$ with $i=1,...,8$. We get:
\begin{eqnarray*}
\abs{I_1} \leq R \psi_{+} \norm{\partial_t C_h^{\tau}}_{L^2(\Omega_T)} \Bigl( \sum \limits_{n=1}^{T/\tau} \int_{t_{n-1}}^{t_n}  \norm{w -w_h}^2_{L^2(\Omega)} \dt \Bigr)^{\frac{1}{2}}.
\end{eqnarray*}
According to \eqref{majorationCntau2} from Lemme \ref{lem:majorationCntau} and the proprieties of the projection operator \eqref{estimationsprojection}, we get
$$\abs{I_1} \leq C  \Bigl( \sum \limits_{n=1}^{T/\tau} \int_{t_{n-1}}^{t_n}  \norm{w -w_h}^2_{L^2(\Omega)} \dt \Bigr)^{\frac{1}{2}} \leq C h \norm{w}_{L^2(0,T;H^1(\Omega))} \leq C h \rightarrow  0 \text{ lorsque $h\rightarrow 0$}$$
because  $w \in L^2(0,T;H^1(\Omega))$. \\
As $t_n = t_{n-1} + \tau$, we get $Z_h^{\tau} - z_h^{\tau} = (t-t_n)(z_h^n - z_h^{n-1})/ \tau$ and $\abs{{V_C}_h^{\tau} - {v_c}_h^n} \leq C \abs{{v_c}_h^n - {v_c}_h^{n-1}}$. Thus, after a integration by parts (with no boundary term because $w \in H^1_0(\Omega)$): 
\begin{eqnarray*}
\abs{I_2} \leq \Bigl( \sum \limits_{n=1}^{T/\tau} \tau \norm{S(v_h)^{\frac{1}{2}} ({v_c}_h^{n} -{v_c}_h^{n-1})}^2_{L^2(\Omega)} \dt \Bigr)^{\frac{1}{2}} \Bigl( \sum \limits_{n=1}^{T/\tau} \int_{t_{n-1}}^{t_n} \norm{\nabla w}^2_{L^2(\Omega)} \dt \Bigr)^{\frac{1}{2}}.
\end{eqnarray*}
According to \eqref{stabilite1}, we get
\begin{eqnarray*}
\abs{I_2} \leq \tau^{\frac{1}{2}} \norm{(S_m + \alpha_L \abs{v_h})^{\frac{1}{2}}}_{L^{\infty}(\Omega)} C \Bigl(1+\norm{v}_{L^{\infty}(\Omega)}+\norm{v}^{2}_{L^{\infty}(\Omega)} \Bigr)^{\frac{1}{2}} \norm{\nabla w}_{(L^2(\Omega_T))^N}.
\end{eqnarray*}
As $v_h \rightarrow v$ strongly in $L^2(\Omega)$, we can pass to the limit. We get
$$  \abs{I_2} = C \tau^{\frac{1}{2}} \Bigl(1+\norm{v}_{L^{\infty}(\Omega)}+\norm{v}^{2}_{L^{\infty}(\Omega)} \Bigr)^{\frac{1}{2}} \norm{w}_{L^2(0,T;H^1(\Omega))} \leq C \tau^{\frac{1}{2}}  \rightarrow 0 \text{ when $\tau \rightarrow 0$.}$$
According to \eqref{stabilite1} and \eqref{estimationsprojection},
\begin{align*}
\abs{I_3} &\leq \Bigl(C  \sum \limits_{n=1}^{T/\tau} \tau \norm{S(v_h)^{\frac{1}{2}} ({v_c}_h^{n} -{v_c}_h^{n-1})}^2_{L^2(\Omega)} \dt \Bigr)^{\frac{1}{2}} \Bigl( \sum \limits_{n=1}^{T/\tau} \int_{t_{n-1}}^{t_n} \norm{\nabla (w-w_h)}^2_{L^2(\Omega)} \dt \Bigr)^{\frac{1}{2}} &&\\
&\leq \tau^{\frac{1}{2}} \norm{(S_m + \alpha_L \abs{v_h})^{\frac{1}{2}}}_{L^{\infty}(\Omega)} C\Bigl(1+\norm{v}_{L^{\infty}(\Omega)}+\norm{v}^{2}_{L^{\infty}(\Omega)} \Bigr)^{\frac{1}{2}}  \norm{w-w_h}_{L^2(0,T;H^1(\Omega))}  &&\\
&\leq C \tau^{\frac{1}{2}} \rightarrow 0 \text{ when $h \rightarrow 0$ and $\tau \rightarrow 0$}&&
\end{align*}
because $w-w_h \in L^2(0,T;H^1(\Omega))$. \\
Similarly, 
\begin{align*}
\abs{I_4} &\leq \Bigl(C  \sum \limits_{n=1}^{T/\tau} \tau \norm{S(v_h)^{\frac{1}{2}} ({v_c}_h^{n} -{v_c}_h^{n-1})}^2_{L^2(\Omega)} \dt \Bigr)^{\frac{1}{2}} \Bigl( \sum \limits_{n=1}^{T/\tau} \int_{t_{n-1}}^{t_n} \norm{\nabla (w_h -w)}^2_{L^2(\Omega)} \dt \Bigr)^{\frac{1}{2}} &&\\
&\leq  \tau^{\frac{1}{2}} \norm{(S_m + \alpha_L \abs{v_h})^{\frac{1}{2}}}_{L^{\infty}(\Omega)} C \Bigl(1+\norm{v}_{L^{\infty}(\Omega)}+\norm{v}^{2}_{L^{\infty}(\Omega)} \Bigr)^{\frac{1}{2}} \norm{w_h -w}_{L^2(0,T;H^1(\Omega))}&& \\
&\leq C \tau^{\frac{1}{2}} \rightarrow 0 \text{ when $h \rightarrow 0$ and $\tau \rightarrow 0$.}&&
\end{align*}
For the isoterm, we get
\begin{eqnarray*}
\abs{I_5} \leq \Bigl( C \sum \limits_{n=1}^{T/\tau} \tau \norm{r(C_h^{\tau})-r(c_h^{n-1})}^2_{L^2(\Omega)} \Bigr)^{\frac{1}{2}} \Bigl( \sum \limits_{n=1}^{T/\tau} \int_{t_{n-1}}^{t_n} \norm{w}^2_{L^2(\Omega)} \dt \Bigr)^{\frac{1}{2}}.
\end{eqnarray*}
As $r$ is supposed to be derivable with a bounded derivate, we get 
\begin{align*}
\abs{I_5} &\leq \tau^{\frac{1}{2}} \norm{r'}_{L^{\infty}(\R)} \Bigl( \sum \limits_{n=1}^{T/\tau}  \norm{C_h^{\tau} -c_h^{n-1}}^2_{L^2(\Omega)}  \Bigr)^{\frac{1}{2}} \Bigl( \sum \limits_{n=1}^{T/\tau} \int_{t_{n-1}}^{t_n} \norm{w}^2_{L^2(\Omega)} \dt \Bigr)^{\frac{1}{2}} &&\\
&\leq C \tau^{\frac{1}{2}} \norm{r'}_{L^{\infty}(\R)} \Bigl( \sum \limits_{n=1}^{T/\tau} \norm{c_h^{n} -c_h^{n-1}}^2_{L^2(\Omega)} \dt \Bigr)^{\frac{1}{2}} \norm{w}_{L^2(\Omega_T)}.&&
\end{align*}
According to \eqref{stabilite1}, we get
\begin{align*}
\abs{I_5} &\leq C \tau^{\frac{1}{2}} \norm{r'}_{L^{\infty}(\R)} \Bigl( \tau \Bigl(1+\norm{v}_{L^{\infty}(\Omega)}+\norm{v}^{2}_{L^{\infty}(\Omega)} \Bigr) \Bigr)^{\frac{1}{2}}  \norm{w}_{L^2(\Omega_T)}&& \\
&\leq C \tau \norm{r'}_{L^{\infty}(\R)}   \Bigl(1+\norm{v}_{L^{\infty}(\Omega)}+\norm{v}^{2}_{L^{\infty}(\Omega)} \Bigr)^{\frac{1}{2}} \norm{w}_{L^2(\Omega_T)} &&\\
&\leq C \tau \rightarrow 0 \text{ when $\tau \rightarrow 0$.}&&
\end{align*}
Similarly, 
\begin{align*}
\abs{I_6}& \leq \Bigl( C \sum \limits_{n=1}^{T/\tau} \tau \norm{r(C_h^{\tau})}^2_{L^2(\Omega)} \dt \Bigr)^{\frac{1}{2}} \Bigl( \sum \limits_{n=1}^{T/\tau} \int_{t_{n-1}}^{t_n} \norm{w -w_h}^2_{L^2(\Omega)} \dt \Bigr)^{\frac{1}{2}} &&\\
&\leq C h \norm{w}_{H^1(\Omega)} \leq C h \rightarrow 0 \text{ when $h \rightarrow 0$,}&&
\end{align*}
\begin{align*}
\abs{I_7} &\leq \Bigl( C \sum \limits_{n=1}^{T/\tau} \tau \norm{r(C_h^{\tau}) -r(c_h^n)}^2_{L^2(\Omega)} \dt \Bigr)^{\frac{1}{2}} \Bigl( \sum \limits_{n=1}^{T/\tau} \int_{t_{n-1}}^{t_n} \norm{w_h -w}^2_{L^2(\Omega)} \dt \Bigr)^{\frac{1}{2}} &&\\
&\leq C \tau \norm{r'}_{L^{\infty}(\R)}  \Bigl(1+\norm{v}_{L^{\infty}(\Omega)}+\norm{v}^{2}_{L^{\infty}(\Omega)} \Bigr)^{\frac{1}{2}}  \norm{w_h -w}_{L^2(\Omega_T)}&&\\
& \leq C\tau \rightarrow 0 \text{ when $h \rightarrow 0$, and $\tau \rightarrow 0$}&&
\end{align*}
and
\begin{align*}
\abs{I_8} \leq \norm{p+\gamma}_{L^2(\Omega)} \Bigl( \sum \limits_{n=1}^{T/\tau} \int_{t_{n-1}}^{t_n} \norm{w_h -w}^2_{L^2(\Omega)} \dt \Bigr)^{\frac{1}{2}} & \leq C h \norm{w}_{H^1(\Omega)} &&\\
&\leq C h \rightarrow 0 \text{ when $h \rightarrow 0$.}&&
\end{align*}

We now focus on the following equation 
\begin{align}
\label{derniere}
&\int_{0}^{T} (S(v_h)^{-\frac{1}{2}}\psi^{-1} {V_C}_h^{\tau},u) \dt - \int_{0}^{T} (C_h^{\tau},\divg u) \dt - \int_{0}^{T} (S(v_h)^{-1}\psi^{-1} v_h C_h^{\tau},u) \dt  \nonumber &&\\
&=\sum \limits_{n=1}^{T/\tau} \int_{t_{n-1}}^{t_n} (S(v_h)^{-\frac{1}{2}}\psi^{-1} ({V_C}_h^{\tau} - {v_c}^n),u) \dt + \sum \limits_{n=1}^{T/\tau} \int_{t_{n-1}}^{t_n} (S(v_h)^{-\frac{1}{2}}\psi^{-1} {v_c}_h^n,u-u_h) \dt  \nonumber &&\\
&+ \sum \limits_{n=1}^{T/\tau} \int_{t_{n-1}}^{t_n} (c_h^n-C_h^{\tau},\divg u) \dt  + \sum \limits_{n=1}^{T/\tau} \int_{t_{n-1}}^{t_n} (c^n,\divg (u_h -u)) \dt \nonumber &&\\
&+ \sum \limits_{n=1}^{T/\tau} \int_{t_{n-1}}^{t_n} (v_h(c_h^n-C_h^{\tau}),u) \dt + \sum \limits_{n=1}^{T/\tau} \int_{t_{n-1}}^{t_n} (v_h c_h^n,u_h -u) \dt &&
\end{align}
for all $u\in L^2(0,T;H^2(\Omega))$ and $u_h = \Pi_h (u)$. The left member converges to the desired limit. Indeed, 
\begin{align*}
 \int_{0}^{T} (S(v_h)^{-1}\psi^{-1} v_h C_h^{\tau},u) \dt =& \int_{0}^{T} (S(v_h)^{-1}\psi^{-1} v C_h^{\tau},u) \dt &&\\
 &+ \int_{0}^{T} (S(v_h)^{-1}\psi^{-1} (v_h -v)C_h^{\tau},u) \dt.&&
 \end{align*}
As $v \in L^{\infty}(\Omega)$, the first term of the right member converges to the desired limit thanks to the strong convergence of$v_h$ to $v$ in $L^2(\Omega_T)$ and almost everywhere. Similarly, the second right term removed because $v_h -v \in (L^{\infty}(\Omega))^N$. \\ 
It remains to prove that the right member of \eqref{derniere} removed at the limit. We denoted the terms by $I_i$, $i=1,...,6$. We get, 
\begin{align*} 
\abs{I_1} &\leq C \norm{u}_{L^2(0,T;L^2(\Omega))} \Bigl( \sum \limits_{n=1}^{T/\tau} \tau \norm{S(v_h)^{-\frac{1}{2}} \psi^{-1} ({v_c}_h^n - {v_c}_h^{n-1})}^2_{L^2(\Omega_T)} \Bigr)^{\frac{1}{2}} &&\\
&\leq  \norm{u}_{L^2(0,T;L^2(\Omega))}  \tau^{\frac{1}{2}} C \Bigl(1+\norm{v}_{L^{\infty}(\Omega)}+\norm{v}^{2}_{L^{\infty}(\Omega)} \Bigr)^{\frac{1}{2}} S_m^{-\frac{1}{2}} \psi_{-}^{-1} &&\\
&\leq C \tau^{\frac{1}{2}} \rightarrow 0 \text{ when $\tau \rightarrow 0$,}&&
\end{align*}
because $v_h$ is bounded in $L^{\infty}(\Omega)$ and thanks to \eqref{stabilite1}. Moreover, according to \eqref{stabilite1}, 
\begin{align*} 
\abs{I_2} &\leq \Bigl(\sum \limits_{n=1}^{T/\tau} \tau \norm{S(v_h)^{-\frac{1}{2}} \psi^{-1} {v_c}_h^n }^2_{L^2(\Omega)} \Bigr)^{\frac{1}{2}} \Bigl(\sum \limits_{n=1}^{T/\tau}  \norm{u-u_h}^2_{L^2(\Omega)} \Bigr)^{\frac{1}{2}} &&\\
&\leq \tau^{\frac{1}{2}} \psi_{-}^{-1} S_m^{-\frac{1}{2}} \sup \norm{{v_c}^n_h}_{L^2(\Omega)} \Bigl(  \sum \limits_{n=1}^{T/\tau} 1 \Bigr)^{\frac{1}{2}} \norm{u-u_h}_{L^2(\Omega)} &&\\
&\leq S_m^{-1} \psi_{-}^{-1}  \Bigl(1+\norm{v}_{L^{\infty}(\Omega)}+\norm{v}^{2}_{L^{\infty}(\Omega)} \Bigr)^{\frac{1}{2}} \norm{u-u_h}_{L^2(\Omega)} &&\\
&\leq C h \norm{u}_{H^1(\Omega)} \leq Ch.&&
\end{align*}
Thus
$$  \abs{I_2}  \rightarrow  0 \text{ when $h \rightarrow 0$.}$$
Similarly, also thanks to \eqref{stabilite1}, 
\begin{align*} 
\abs{I_3} &\leq \Bigl(\sum \limits_{n=1}^{T/\tau} \tau \norm{c_h^n - c_h^{n-1} }^2_{L^2(\Omega)} \Bigr)^{\frac{1}{2}} \Bigl(\sum \limits_{n=1}^{T/\tau} \norm{\divg(u)}^2_{L^2(\Omega)} \Bigr)^{\frac{1}{2}} &&\\
&\leq \tau C\Bigl(1+\norm{v}_{L^{\infty}(\Omega)}+\norm{v}^{2}_{L^{\infty}(\Omega)} \Bigr)^{\frac{1}{2}} \Bigl(\sum \limits_{n=1}^{T/\tau} \norm{\divg(u)}^2_{L^2(\Omega)} \Bigr)^{\frac{1}{2}} &&\\
&\leq C \tau^{\frac{1}{2}} \rightarrow  0 \text{ when $\tau \rightarrow 0$,}&&
\end{align*}
because $u \in L^2(0,T;H(\divg;\Omega))$. \\
According to \eqref{CFLc} and \eqref{estimationsprojection}, we get
\begin{align*} 
\abs{I_4} &\leq \Bigl(\sum \limits_{n=1}^{T/\tau} \tau \norm{c_h^n }^2_{L^2(\Omega)} \Bigr)^{\frac{1}{2}} \Bigl(\sum \limits_{n=1}^{T/\tau}  \norm{\divg(u-u_h)}^2_{L^2(\Omega)} \Bigr)^{\frac{1}{2}}&& \\
&\leq  C h \norm{u}_{H^2(\Omega)}  \leq C h \rightarrow  0 \text{ when $h \rightarrow 0$.}&&
\end{align*}
As $\norm{v_h}_{(L^{\infty}(\Omega))^N} \leq C$, thanks to \eqref{stabilite1} we get
\begin{align*} 
\abs{I_5} &\leq C \Bigl(\sum \limits_{n=1}^{T/\tau} \tau \norm{c_h^n -c_h^{n-1}}^2_{L^2(\Omega)} \Bigr)^{\frac{1}{2}} \norm{u}_{L^2(0,T;L^2(\Omega))}&& \\
&\leq \tau C \Bigl(1+\norm{v}_{L^{\infty}(\Omega)}+\norm{v}^{2}_{L^{\infty}(\Omega)} \Bigr)^{\frac{1}{2}}  \norm{u}_{L^2(0,T;L^2(\Omega))} \leq C \tau  \rightarrow  0 \text{ when $\tau \rightarrow 0$.}&&
\end{align*}
Similarly, 
\begin{align*} 
\abs{I_6} &\leq C \Bigl(\sum \limits_{n=1}^{T/\tau} \tau \norm{c_h^n}^2_{L^2(\Omega)} \Bigr)^{\frac{1}{2}}  \Bigl(\sum \limits_{n=1}^{T/\tau} \norm{u-u_h}^2_{L^2(\Omega)} \Bigr)^{\frac{1}{2}} &&\\
&\leq C h \norm{u}^2_{H^1(\Omega)}  \leq C h \rightarrow  0 \text{ when $h \rightarrow 0$.}&&
\end{align*}
\end{proof}

\begin{thm}
\label{thm:cvcompletementdiscret}
The sequence $(C^{\tau}_h,{V_C}^{\tau}_h)$ converges to the solution $(c,v_c)$ de \eqref{problemeP}.
\end{thm}
\begin{proof}
By using the result of the previous lemma and the uniqueness of the solution, we get $(c^{\#},v_c^{\#})=(c,v_c)$ and the sequence  $(C^{\tau}_h,{V_C^{\tau}}_h)$ converges to this limit. 
\end{proof}


\end{document}